\documentclass[preprint,EJP]{ejpecp}

\usepackage[T1]{fontenc}
\usepackage[utf8]{inputenc}

\SHORTTITLE{Interval-Activation Frog Model with Random Survival Parameters}

\TITLE{Extinction and Survival in an Interval-Activation Frog Model on
  \texorpdfstring{\(\mathbb{Z}\)}{Z} with Random Survival Parameters and
  Symmetric Random Walks}

\AUTHORS{%
  Gustavo~Oshiro~de~Carvalho\footnote{Universidade Estadual de Campinas,
    Brazil. \EMAIL{godc@ime.unicamp.br}}
  \and
  Fábio~Prates~Machado\footnote{Universidade de São Paulo,
    Brazil. \EMAIL{fmachado@ime.usp.br}}
  \and
  José~Hermenegildo~Ramírez-González\footnote{Universidade de São Paulo,
    Brazil. Corresponding author.
    \EMAIL{hermenegildo.ramirez@usp.br}}}

\AMSSUBJ{60K35}
\AMSSUBJSECONDARY{60F05; 60G50; 60G52}

\ABSTRACT{We study an interval-activation frog model on \(\mathbb Z\). At each \(x\in\mathbb Z\), let \(\eta_x\in\mathbb N_0:=\{0,1,2,\ldots\}\) be the number of frogs initially placed at \(x\), where \((\eta_x)_{x\in\mathbb Z}\) is i.i.d.\ with \(0<\E[\eta_0]<\infty\); frogs at the origin are initially active and all others are sleeping. Each frog \((x,i)\), \(1\le i\le\eta_x\), is assigned a survival parameter \(\pi_{x,i}\in(0,1)\), a lifetime \(L_{x,i}\in\mathbb N_0\) and a random walk \(S_m^{x,i}=x+\sum_{j=1}^m\xi_j^{x,i}\), where the increments \(\xi_j^{x,i}\) are i.i.d.\ copies of a symmetric, integer-valued random variable \(\xi_1\). The survival parameters \((\pi_{x,i})_{x\in\mathbb Z,\,i\ge1}\) are i.i.d.\ with common law \(\nu\) on \((0,1)\) and the lifetimes \((L_{x,i})_{x\in\mathbb{Z},i\ge 1}\) are i.i.d. satisfying \(\PP(L_{x,i}\ge k\mid\pi_{x,i}=p)=p^k\), \(k\in\mathbb N_0\), representing the case in which each active frog survives before each jump with probability \(\pi_{x,i}\).
While active and alive, each frog follows its random walk, and every jump activates all sleeping frogs at the integer sites between its endpoints.
 Let \(D^\to\) have the same law as \(D_{x,i}^\to:=\max_{0\le m\le L_{x,i}}(S_m^{x,i}-x)\). We obtain survival and extinction criteria for the interval-activation frog model by studying the tail behavior of \(D^\to\). If \(\PP(|\xi_1|\ge n)\sim n^{-\alpha}L_\xi(n)\) as \(n\to\infty\) with $L_\xi$ slowly varying, then the model survives with positive probability for $0<\alpha<1$, whereas for $\alpha=1$ both survival and almost sure extinction are possible, depending on the remaining parameters. For \(1<\alpha<2\), assume \(\PP(|\xi_1|>n)\sim c_\xi n^{-\alpha}\) as \(n\to\infty\); in the finite-variance case assume \(\E[\xi_1]=0\) and \(\Var(\xi_1)=\sigma^2\in(0,\infty)\), and set \(r=\alpha\) and \(r=2\), respectively. If \(\nu\) has density \(f_\pi(u)\sim(1-u)^{\beta-1}\ell((1-u)^{-1})\) as \(u\uparrow1\) with $\ell$ slowly varying, then, for \(0<\beta<1\), \(n\PP(D^\to\ge n)\sim C_\beta n^{1-r\beta}\ell(n^r)\), with explicit \(C_\beta\). Consequently, we obtain the sharp off-critical threshold of \(\beta_c=1/r\): survival holds with positive probability for \(\beta<1/r\), extinction holds almost surely for \(\beta>1/r\), and explicit sufficient conditions on the critical line \(\beta=1/r\) leave a factor-four gap.}

\newcommand{\PP}{\mathbb{P}}
\newcommand{\E}{\mathbb{E}}
\newcommand{\N}{\mathbb{N}}
\newcommand{\R}{\mathbb{R}}
\newcommand{\Var}{\operatorname{Var}}
\newcommand{\Disc}{\operatorname{Disc}}

\hypersetup{
  pdftitle={Extinction and Survival in an Interval-Activation Frog Model on Z with Random Survival Parameters and Symmetric Random Walks},
  pdfauthor={José Hermenegildo Ramírez-González, Gustavo Oshiro de Carvalho, Fábio Prates Machado},
  pdfsubject={Survival and extinction in an interval-activation frog model},
  pdfkeywords={Frog model, interval activation, random survival parameter, geometric lifetime, regular variation, stable domain of attraction}
}

\begin{document}

\section{Introduction and setting}\label{sec:intro}

In the frog model, frogs are initially placed at the vertices of a graph. The
frogs at a distinguished root are active and all others are sleeping. Active
frogs perform independent random walks, and a sleeping frog becomes active when
an active trajectory reaches its site. The basic question is whether at least
one frog remains active at every time or eventually no active frogs remain. Since the foundational work
\cite{Alves2002}, the model has become a standard framework for studying
survival and extinction in spatially distributed systems.

The behavior of the frog model depends strongly on the geometry, the motion,
and the lifetime mechanism. On trees and higher-dimensional lattices, genuine
phase transitions arise and the associated critical parameters can often be
estimated or bounded; see, for example,
\cite{Fontes2004,Lebensztayn2005,Lebensztayn2019,Lebensztayn2020,Gallo2023}.
On \(\mathbb Z^d\), \(d\ge2\), a phase transition with respect to the
survival parameter was established in \cite{Alves2002}. In contrast, on
$\mathbb Z$, geometric lifetimes with a fixed survival parameter lead to
almost-sure extinction in the standard model \cite{Alves2002}. This rigidity
has motivated variants with drifted walks or spatially heterogeneous
parameters \cite{Bertacchi2014,Gantert2009}.

Assigning an independent random survival parameter to each frog provides such a mechanism. In
\cite{CarvalhoMachado2025}, each frog is assigned an independent
$\pi\in(0,1)$ and, conditional on $\pi=p$, has geometric lifetime tail $p^k$.
For Beta laws, the extinction--survival threshold occurs at $\beta=1/2$.
This result was extended in \cite{JGF} to survival-parameter densities satisfying
\begin{equation}\label{eq:edge-behavior-pi}
f_\pi(u)\sim (1-u)^{\beta-1}\ell\!\left(\frac{1}{1-u}\right),
\qquad u\uparrow1,
\end{equation}
where $\beta>0$ and $\ell$ is slowly varying at infinity. Throughout, slowly
varying functions are assumed measurable and eventually positive.
Recall that $\ell:(0,\infty)\to(0,\infty)$ is slowly varying if
\[
\lim_{t\to\infty}\frac{\ell(ct)}{\ell(t)}=1
\qquad\text{for every }c>0.
\]
We use the standard theory of regular variation, including Potter bounds, from
\cite{BinghamGoldieTeugels1987}, and write $f(t)\sim g(t)$ when
$f(t)/g(t)\to1$. For the simple symmetric nearest-neighbour walk, the result in
\cite{JGF} gives extinction for $\beta>1/2$, survival with positive probability
for $\beta<1/2$, and a critical regime governed by $\ell$.

A complementary extension retains nearest-neighbour motion and changes the
lifetime law. In \cite{JGF2}, conditional on $\pi=p$, the lifetime satisfies
\[
\PP(L\ge k\mid\pi=p)=p^{k^\gamma},
\qquad k\in\{0,1,2,\ldots\},
\]
with $\gamma>0$, and the sharp off-critical threshold is
$\beta_c=1/(2\gamma)$. The same random discrete Weibull mechanism was studied
for biased nearest-neighbour walks in \cite{RamirezMachadoBiasedWeibull}; the
linear first-passage scale produced by the drift changes the threshold to
$\beta_c=1/\gamma$. These results show that the critical exponent reflects both
the lifetime tail and the spatial scale of the underlying motion.

The present paper keeps geometric lifetimes and broadens the class of symmetric
motions. Let \((\xi_j)_{j\ge1}\) be an i.i.d.\ sequence of symmetric,
integer-valued increments and set \(S_0=0\) and
\(S_m=\sum_{j=1}^m\xi_j\). We consider three motion regimes. For
\(0<\alpha\le1\), we assume
\(\PP(|\xi_1|\ge n)\sim n^{-\alpha}L_\xi(n)\), where \(L_\xi\) is slowly varying. For \(1<\alpha<2\), we assume
\(\PP(|\xi_1|>n)\sim c_\xi n^{-\alpha}\) for some \(c_\xi>0\) as \(n\to\infty\). We also treat the finite-variance case, where \(\E[\xi_1]=0\) and
\(\Var(\xi_1)=\sigma^2\in(0,\infty)\). In particular, this regime
includes the case where
\(\PP(|\xi_1|\ge n)\sim n^{-\alpha}L_\xi(n)\), with
\(\alpha>2\) and \(L_\xi\) slowly varying.

Because genuine long jumps may skip intermediate sites, we use an
\emph{interval-activation rule}: a displacement from \(x\) to \(y\) activates
every sleeping frog at the integer sites in the interval with endpoints \(x\)
and \(y\). This convention makes the region explored by a frog a complete
integer interval and permits the firework comparisons stated in
Theorem~\ref{thm:frog-criterion}. It is part of the model definition and is not
a consequence of the endpoint-only long-jump frog model. The two activation
rules coincide for nearest-neighbour walks.

We now describe the argument and the results precisely. Let \(\pi\) have law
\(\nu\). Conditional on \(\pi=p\), let \(L\) satisfy
\(\PP(L\ge k\mid\pi=p)=p^k\), and let the random walk \((S_n)_{n\ge0}\),
started from zero, be independent of \(L\). Define
\[
D^\rightarrow:=\max_{0\le m\le L}S_m,
\qquad
D^*:=\max_{0\le m\le L}|S_m|,
\qquad
\tau_n^+:=\inf\{m\ge0:S_m\ge n\}.
\]
Conditional on \(\pi=p\), the geometric lifetime gives the fundamental identity
\[
\PP(D^\rightarrow\ge n\mid\pi=p)=\E[p^{\tau_n^+}],
\]
where \(p^\infty:=0\). The firework couplings in
Theorem~\ref{thm:frog-criterion} reduce survival and extinction to the
asymptotic behavior of \(n\PP(D^\rightarrow\ge n)\) and
\(n\PP(D^*\ge n)\). For \(0<\alpha<1\), a one-step lower bound yields
\(n\PP(D^\rightarrow\ge n)\to\infty\); the same holds at \(\alpha=1\)
under \(L_\xi(n)\to\infty\), giving survival with positive probability
(Proposition~\ref{prop:survival-alpha-le-one}). When $\alpha=1$, additional criteria for both survival and extinction are given in Proposition~\ref{prop_regularly_alpha1}, based on the single big jump principle for subexponential random variables.

For \(1<\alpha<2\), the rescaled walk converges to a symmetric strictly
\(\alpha\)-stable L\'evy process, while in the finite-variance regime it
converges to a Brownian motion. We prove convergence of the corresponding
first-passage times at scale \(n^r\), where \(r=\alpha\) in the stable regime
and \(r=2\) in the finite-variance regime. Combining this convergence with
uniform small-time estimates and the edge asymptotic
\eqref{eq:edge-behavior-pi}, Proposition~\ref{prop:An-order} proves, for
\(0<\beta<1\),
\[
n\PP(D^\rightarrow\ge n)
\sim C_\beta n^{1-r\beta}\ell(n^r),
\]
with an explicit constant \(C_\beta\). Theorem~\ref{thm:survival-extinction}
therefore gives the sharp off-critical threshold \(\beta_c=1/r\): survival
has positive probability for \(\beta<1/r\), whereas extinction occurs almost
surely for \(\beta>1/r\). On the critical line \(\beta=1/r\), the exact
asymptotic is \(\kappa\ell(n^r)\), with an explicit motion-dependent constant
\(\kappa\); the firework comparisons yield sufficient survival and extinction
conditions whose constants differ by a factor of four.

The paper is organized as follows. Section~\ref{modeln} defines the model and
the displacements of a single frog. Section~\ref{mainr} states the
comparison theorem and the phase-transition results, whose proofs are given in
Section~\ref{proofmain}. Section~\ref{auxsec} contains the functional-limit and
first-passage estimates used in those proofs.

\section{Model and notation on \texorpdfstring{\ensuremath{\mathbb{Z}}}{Z}}\label{modeln}

Let $\mathbb{N}=\{1,2,3,\dots\}$ and $\mathbb{N}_0=\mathbb{N}\cup\{0\}$. We write $\mathbf 1_A$ for the indicator of an event $A$, $U\stackrel d=V$ for equality in distribution, and $X_n\Rightarrow X$ for convergence in distribution. For each $x\in\mathbb{Z}$, let $\eta_x\in\mathbb{N}_0$ denote the initial number of frogs at $x$, and index frogs by pairs $(x,i)$ with $1\le i\le \eta_x$.

The families
\[
\{\eta_x\}_{x\in\mathbb{Z}},\qquad
\{\pi_{x,i}\}_{x\in\mathbb{Z},\,i\in\mathbb{N}},\qquad
\Big\{(S_n^{x,i})_{n\in\mathbb{N}_0}\Big\}_{x\in\mathbb{Z},\,i\in\mathbb{N}}
\]
are mutually independent, where $\{\eta_x\}_{x\in\mathbb{Z}}$ is an i.i.d.\ sequence, $\{\pi_{x,i}\}_{x\in\mathbb{Z},\,i\in\mathbb{N}}$ is an i.i.d.\ family with common law $\nu$ on $(0,1)$, and
\[
\Big\{(S_n^{x,i})_{n\in\mathbb{N}_0}\Big\}_{x\in\mathbb{Z},\,i\in\mathbb{N}}
\]
is a family of independent random walks on $\mathbb{Z}$ such that $S_0^{x,i}=x$ for every $(x,i)$ and all walks have common one-step distribution $\mu$. Let $S=(S_n)_{n\in\mathbb N_0}$ be a random walk started from $S_0=0$ with one-step distribution $\mu$; then $(S_n^{x,i}-x)_{n\in\mathbb N_0}$ is an independent copy of $S$ for every $(x,i)$. We write $\pi$ for a random variable with law $\nu$ and $\eta$ for a random variable with the same law as $\eta_0$. Thus, $\eta_x$ represents the initial number of frogs at site $x$, $\pi_{x,i}$ is the survival parameter assigned to frog $(x,i)$, and $(S_n^{x,i})_{n\in\mathbb{N}_0}$ is the trajectory followed by that frog.

For the phase-transition results in Subsection~\ref{sec:alpha-ge-one}, we assume that the law $\nu$ admits a density $f_\pi$ satisfying \eqref{eq:edge-behavior-pi}. The general comparison theorem and Proposition~\ref{prop:survival-alpha-le-one} are stated under their own weaker assumptions and do not require this density condition.

For each $(x,i)$, let $L_{x,i}\in\mathbb{N}_0$ denote the lifetime of frog $(x,i)$. Given $\pi_{x,i}=p$, the random variable $L_{x,i}$ is geometric with tail
\[
\PP(L_{x,i}\ge k\mid \pi_{x,i}=p)=p^k,
\qquad k\in\mathbb{N}_0.
\]
Conditional on $\{\pi_{x,i}\}_{x\in\mathbb{Z},\,i\in\mathbb{N}}$, the family $\{L_{x,i}\}_{x\in\mathbb{Z},\,i\in\mathbb{N}}$ is independent, and is independent of $\{\eta_x\}_{x\in\mathbb{Z}}$ and of
\[
\Big\{(S_n^{x,i})_{n\in\mathbb{N}_0}\Big\}_{x\in\mathbb{Z},\,i\in\mathbb{N}}.
\]
Hence, conditional on $\pi_{x,i}=p$, frog $(x,i)$ survives for at least $k$ steps with probability $p^k$. Equivalently, $L_{x,i}$ is the number of steps executed by the frog before its death event.

Time is discrete and updates are synchronous. At time $0$, all frogs located at the origin are active, whereas all frogs at the remaining sites are inactive. During the transition from time $t$ to time $t+1$, each frog that is active at time $t$ first dies with probability $1-\pi_{x,i}$; if it survives, it performs the next step of its associated random walk. A successful displacement from a site $x$ to a site $y$ activates, at time $t+1$, every sleeping frog initially present at an integer site between $x$ and $y$ (see Figure~\ref{fig:interval-activation}). Frogs activated at time $t+1$ make their first death-or-move attempt during the transition from $t+1$ to $t+2$. Once activated, frogs evolve independently according to the same mechanism. In particular, once no active frogs remain, no further frog can be activated.

\begin{figure}[t]
  \centering
  \includegraphics[width=.94\textwidth]{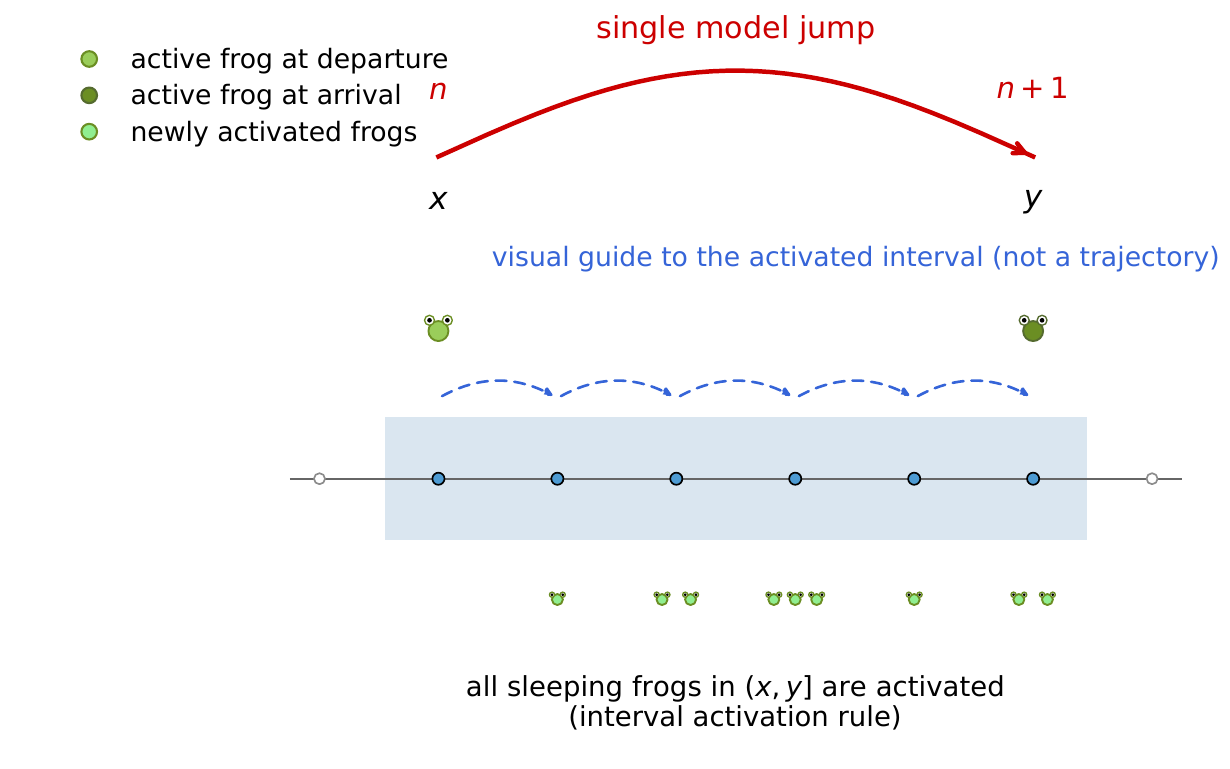}
  \caption{Schematic of the interval-activation convention. A displacement
  from \(x\) to \(y\) activates all sleeping frogs at the integer sites in the
  interval with endpoints \(x\) and \(y\). The dashed nearest-neighbour path is
  only a visual representation of the intervening interval and is not part of
  the stochastic dynamics.}
  \label{fig:interval-activation}
\end{figure}

We denote the resulting system by $\mathrm{FM}(\mathbb{Z},\pi,\eta,S)$,
where $\pi$ denotes a survival parameter with law $\nu$, $\eta$ has the same law as the initial number $\eta_0$ of frogs at a site, and $S$ denotes the walk started from zero. The process depends only on these three laws. A realization \emph{survives}
if at every integer time there is at least one active frog; otherwise it becomes
extinct. Because every jump has finite length and every site contains finitely
many frogs almost surely, only finitely many sites can be activated by any fixed
time. Indeed, this follows inductively over the discrete time steps: finitely
many active frogs make finitely many finite jumps, each of which intersects only
finitely many sites.

\subsection*{Heuristic motivation for the interval-activation rule}
The interval-activation rule is a modelling convention in which each
displacement represents an interaction covering the whole integer segment
between its endpoints. No auxiliary microscopic nearest-neighbour trajectory
is assumed. The rule is distinct from the endpoint-only frog model and makes
the region explored by one frog a complete interval, enabling comparison with
right-going and bidirectional firework processes.

We now specify the increments of the walk $S$, which starts from zero. Let
\((\xi_k)_{k\ge1}\) be an i.i.d.\ sequence of symmetric, integer-valued random
variables such that
\[
S_n=\sum_{k=1}^n \xi_k,\qquad n\ge1.
\]
Thus
\[
\xi_1\overset{d}{=}-\xi_1.
\]
We henceforth take $\mu$ to be the law of $\xi_1$; equivalently, for every
$(x,i)$, the shifted process $(S_n^{x,i}-x)_{n\ge0}$ is an independent copy
of $S$. No aperiodicity or span assumption is imposed; in particular, the
simple symmetric nearest-neighbour walk is included in the finite-variance
regime.

We consider two regimes for the law of $\xi_1$.

\medskip
\noindent
\emph{Heavy-jump regime.}
\begin{itemize}
    \item If $\alpha\in(1,2)$, we assume that there exists a constant $c_\xi>0$ such that
    \begin{equation}\label{eq:tail-xi-lattice}
    \PP(|\xi_1|>n)\sim c_\xi n^{-\alpha},
    \qquad n\to\infty.
    \end{equation}

    \item If $\alpha\in(0,1]$, we assume that there exists a slowly varying function $L_\xi$ such that
    \begin{equation}\label{eq:tail-xi-lattice0}
    \PP(|\xi_1|\ge n)\sim n^{-\alpha}L_\xi(n),
    \qquad n\to\infty.
    \end{equation}
\end{itemize}

\medskip
\noindent
\emph{Finite-variance regime.}
We assume that
\[
\E[\xi_1]=0,
\qquad
\E[\xi_1^2]=\sigma^2\in(0,\infty).
\]

For each frog $(x,i)$, define
\[
D_{x,i}^\rightarrow
:= \max_{0\le n\le L_{x,i}}(S_n^{x,i}-x),\qquad
D_{x,i}^\leftarrow
:= \max_{0\le n\le L_{x,i}}(x-S_n^{x,i}),
\]
and
\[
D_{x,i}^*:=\max\{D_{x,i}^\rightarrow,D_{x,i}^\leftarrow\}.
\]
Under the interval-activation rule, frog $(x,i)$ activates exactly the sites in
\[
[x-D_{x,i}^\leftarrow,\,x+D_{x,i}^\rightarrow]\cap\mathbb{Z}.
\]
Indeed, the intervals traversed by consecutive jumps overlap at their common
endpoint, so their union is the complete integer interval between the minimum
and maximum positions attained by the frog.

We write \(D^\rightarrow\) and \(D^*\) for random variables having the same distributions as \(D_{x,i}^\rightarrow\) and \(D_{x,i}^*\), respectively.

Since activation in \(\mathrm{FM}(\mathbb{Z},\pi,\eta,S)\) is determined by the
range explored by the underlying walk, first-passage times arise naturally in
the analysis. In order to study the probability that a frog reaches distant
sites before dying, we introduce the following notation. For a real-valued
process \(Z=(Z_t)_{t\ge0}\) and \(x>0\), we define
\[
  \tau_x^+(Z):=\inf\{t\ge0:\ Z_t\ge x\}.
\]

In particular, for the underlying walk \(S=(S_m)_{m\in\mathbb N_0}\) and
\(n\in\mathbb N_0\), we write
\[
  \tau_n^+:=\inf\{m\ge0:\ S_m\ge n\},
\]
the first-passage time of \(S\) above level \(n\).

\section{Main results}\label{mainr}

\subsection*{Survival and extinction criterion}

In this subsection we work under the interval-activation convention defined in
Section~\ref{modeln}, for the frog model
$\mathrm{FM}(\mathbb{Z},\pi,\eta,S)$ just defined.

For each $x\in\mathbb{Z}$, there are $\eta_x$ frogs initially located at $x$,
indexed by $\{(x,i):1\le i\le \eta_x\}$. For each such frog, let
$D_{x,i}^{\rightarrow}$ and $D_{x,i}^{*}$ denote, respectively, its maximal
rightward displacement and its maximal displacement in absolute value during
its lifetime. We then define the random radii
\[
I_x^+ := \max_{1\le i\le \eta_x} D_{x,i}^{\rightarrow},
\qquad
I_x^* := \max_{1\le i\le \eta_x} D_{x,i}^{*},
\]
with the convention that the maximum over the empty set is $0$. Write
\[
I^+:=(I_x^+)_{x\in\mathbb Z},
\qquad
I^*:=(I_x^*)_{x\in\mathbb Z}.
\]
Since $\{\eta_x\}_{x\in\mathbb{Z}}$ is an i.i.d.\ family, and
$\{(D_{x,i}^{\rightarrow},D_{x,i}^{*})\}_{x\in\mathbb{Z},\,i\in\mathbb{N}}$
is an i.i.d.\ family independent of $\{\eta_x\}_{x\in\mathbb{Z}}$, it follows
that $\{(I_x^+,I_x^*,\eta_x)\}_{x\in\mathbb{Z}}$ is an i.i.d.\ family.

We now consider the right-going firework process $\mathrm{FW}(I^+)$ and the
bidirectional firework process $\mathrm{BFW}(I^*)$ on $\mathbb{Z}$; see
\cite{JuniorMachadoZuluaga2011,RoySaha2024} and
\cite[Lemma~2.1]{CarvalhoMachado2025}. In both processes the origin is initially
informed and every other site is initially uninformed. In $\mathrm{FW}(I^+)$,
whenever a site $x$ becomes informed, it informs all sites $y$ such that
$x<y\le x+I_x^+$. In $\mathrm{BFW}(I^*)$, whenever a site $x$ becomes informed,
it informs all sites $y$ such that
$x-I_x^*\le y\le x+I_x^*$.

The interval-activation rule associates with each frog the full integer interval determined by its leftmost and rightmost positions. This permits comparison with firework processes, in which the radius attached to each site determines which sites are informed.

\begin{theorem}
\label{thm:frog-criterion}
Assume $0<\E[\eta_0]<\infty$. Let $(D^\rightarrow,D^\leftarrow,D^*)$ have the same distribution as
$(D_{x,i}^{\rightarrow},D_{x,i}^\leftarrow,D_{x,i}^{*})$. Then the following assertions hold.

\begin{itemize}
  \item[(i)] If
  \[
    \limsup_{n\to\infty} n\,\PP(D^*\ge n)
    \;<\;\frac{1}{2\E[\eta_0]},
  \]
  then $\mathrm{FM}(\mathbb{Z},\pi,\eta,S)$ dies out almost surely.

  \item[(ii)] If
  \[
    \liminf_{n\to\infty} n\,\PP(D^\rightarrow\ge n)
    \;>\;\frac{1}{\E[\eta_0]},
  \]
  then $\mathrm{FM}(\mathbb{Z},\pi,\eta,S)$ survives with positive probability.
\end{itemize}
\end{theorem}

By symmetry, \(\{D_{x,i}^*\geq n\}=\{D_{x,i}^\rightarrow\geq n\}\cup\{D_{x,i}^\leftarrow \geq n\}\) and the random variables \(D_{x,i}^\rightarrow\) and \(D_{x,i}^\leftarrow\) are identically distributed. Therefore,
\begin{equation}\label{compara_d}
 \PP(D^\rightarrow\ge n)=\PP(D^\leftarrow\ge n)\ge\frac{1}{2}\PP(D^*\ge n).
\end{equation}
Thus, it suffices to study the asymptotic behavior of
\(\PP(D^{\rightarrow}\ge n)\).

Let $\pi$ have law $\nu$, $S$ be an independent copy of the walk started
from zero, and let $L$ be independent of $S$.
Conditionally on $\pi=p$, the frog survives before each attempted step
with probability $p$. Hence the total number of executed steps before death
has a geometric distribution on $\{0,1,2,\dots\}$ with tail
\[
\PP(L\ge m\mid \pi=p)=p^{\,m},\qquad m\in\mathbb{N}_0.
\]
Recall the upward hitting time
\[
\tau_n^+=\inf\{m\ge0:\ S_m\ge n\}.
\]
Since $D^\rightarrow=\max_{0\le m\le L} S_m$, we have the identity
\[
\{D^\rightarrow\ge n\}=\{\tau_n^+\le L\}.
\]
Using the conditional independence of $L$ and $(S_m)_{m\ge0}$ given $\pi$, and then conditioning on $\tau_n^+$, it follows that
\begin{equation}\label{asymtau}
\PP(D^\rightarrow\ge n\mid \pi=p)
=\E\!\left[\PP(L\ge \tau_n^+\mid \pi=p,\tau_n^+)\right]
=\E\!\left[p^{\tau_n^+}\right],
\end{equation}
where we use the convention $p^\infty=0$ for $0\le p<1$.

In view of Theorem~\ref{thm:frog-criterion} and \eqref{asymtau}, the problem
reduces to studying $\E[p^{\tau_n^+}]$, and hence
$n\PP(D^{\rightarrow}\ge n)$, in order to derive survival and extinction
criteria for $\mathrm{FM}(\mathbb{Z},\pi,\eta,S)$.

\subsection{Heavy-jump regime: \texorpdfstring{$0<\alpha\le1$}{0 < alpha <= 1}}\label{sec:heavy}
\begin{proposition}
\label{prop:survival-alpha-le-one}
Assume that the one-step increment $\xi_1$ is symmetric. If
\[
\liminf_{n\to \infty} n\PP(|\xi_1|\ge n)=\infty,
\]
then the interval-activation frog model survives with positive
probability under the assumptions that
\(\PP(\pi>0)>0\) and \(0<\E[\eta_0]<\infty\).
\end{proposition}

In particular, Proposition~\ref{prop:survival-alpha-le-one} yields survival in the cases where
\begin{equation}\label{eq:cond_nalpha}
\PP(|\xi_1|\ge n)\sim n^{-\alpha}L_\xi(n),\qquad n\to\infty
\end{equation}
with $0<\alpha<1$ and $L_\xi$ slowly varying, as well as in the borderline case $\alpha=1$ under the additional assumption that $L_\xi(n)\to\infty$. Moreover, by additionally assuming that both $\E[\eta]$ and $\E\!\left[\frac{\pi}{1-\pi}\right]$ are finite, we can also establish survival criteria for the case $\alpha=1$ when
\(
0<\liminf_{n\to\infty}L_\xi(n)\le
\limsup_{n\to\infty}L_\xi(n)<\infty.
\)

\begin{proposition}\label{prop_regularly_alpha1}
Assume that the one-step increment $\xi_1$ is symmetric and satisfies \eqref{eq:cond_nalpha} with $\alpha=1$ and $L_\xi$ slowly varying. Suppose that
\(
0<\E[\eta_0]<\infty\) and \(
0<\E[\frac{\pi}{1-\pi}]<\infty.
\)
Then the following items hold:
\begin{itemize}
    \item[(i)] If
    \[
    \liminf_{n\to\infty}L_\xi(n)>
    \frac{4}{\E\!\left[\frac{\pi\eta_0}{1-\pi}\right]},
    \]
    then the interval-activation frog model survives with positive probability.

    \item[(ii)] Assume additionally that
    \(\E[(1-\pi)^{-2}]<\infty.
    \)
    If
    \[
    0<\liminf_{n\to\infty}L_\xi(n)
    \le
    \limsup_{n\to\infty}L_\xi(n)
    <
    \frac{1}{2\E\!\left[\frac{\pi\eta_0}{1-\pi}\right]},
    \]
    then the interval-activation frog model dies out almost surely.
\end{itemize}
\end{proposition}

By performing a change of variables in the corresponding integral, one observes that if $\pi$ has a density satisfying \eqref{eq:edge-behavior-pi}, then the behavior of
\(
\frac{u}{1-u}f_\pi(u)
\)
as $u\uparrow1$ is equivalent to that of
\(
t^{-\beta}\ell(t)
\)
as $t\to\infty$. Consequently,
\(
\E[\frac{\pi}{1-\pi}]<\infty
\)
for $\beta>1$, whereas
\(
\E[\frac{\pi}{1-\pi}]=\infty
\)
for $\beta<1$. The borderline case $\beta=1$ depends on the asymptotic behavior of the slowly varying function $\ell$. By a similar argument,
\(
\E[(1-\pi)^{-2}]<\infty
\)
for $\beta>2$, while
\(
\E[(1-\pi)^{-2}]=\infty
\)
for $\beta<2$, and the case $\beta=2$ again depends on $\ell$.

Moreover, a straightforward coupling argument shows that
\(
\PP(\mathrm{FM}(\mathbb{Z},\pi,\eta,S)\text{ survives})
\)
is nondecreasing with respect to the stochastic order of \(\pi\).
Indeed, if
\(
\PP(\pi'\le x)\ge \PP(\pi\le x)
\)
for every \(x\in(0,1)\), then \(\pi'\) is stochastically dominated
by \(\pi\), and one can couple the models
\(
\mathrm{FM}(\mathbb{Z},\pi',\eta,S)
\)
and
\(
\mathrm{FM}(\mathbb{Z},\pi,\eta,S)
\)
so that they have the same initial number of frogs at every vertex
and the same random walk trajectories, while every frog in the former
model has a lifetime no larger than its counterpart in the latter.

Consequently, when \(\beta<1\), although
\(
\E[\frac{\pi}{1-\pi}]=\infty,
\)
survival for the case \(\alpha=1\) with
\(
\liminf_{n\to\infty}L_\xi(n)>0
\)
can still be established as follows. For \(0<a<1\), define
\[
\pi_a:=\min\{\pi,a\}.
\]
Then \(\pi_a\le\pi\) almost surely and
\[
\E\!\left[\frac{\pi_a}{1-\pi_a}\right]
\le
\frac{a}{1-a}
<
\infty.
\]
Moreover, by the Monotone Convergence Theorem,
\[
\E\!\left[\frac{\pi_a}{1-\pi_a}\right]
\uparrow
\E\!\left[\frac{\pi}{1-\pi}\right]
=
\infty
\qquad\text{as }a\uparrow1.
\]
Since \(\pi_a\) and \(\eta_0\) are independent,
\[
\E\!\left[\frac{\pi_a\eta_0}{1-\pi_a}\right]
\uparrow\infty.
\]
Thus, \(a\) can be chosen sufficiently close to \(1\) so that
the condition in Proposition~\ref{prop_regularly_alpha1}\textup{(i)}
is satisfied. That proposition gives survival for the model with
survival parameter \(\pi_a\), and the coupling above then gives
survival for the original model with survival parameter \(\pi\).

The next subsection treats the stable and finite-variance regimes
and includes, in particular, the case where both
\(
\E[\frac{\pi}{1-\pi}]=\infty
\)
and
\(
\lim_{n\to\infty}n\PP(\xi_1\ge n)=0,
\)
for which a more delicate analysis is required.

\subsection{Stable regime \texorpdfstring{\(1<\alpha<2\)}{1 < alpha < 2} and the finite-variance regime}\label{sec:alpha-ge-one}

In this subsection we treat jointly the case \(\alpha\in(1,2)\) and the
finite-variance case. Since the corresponding results admit a common
formulation, we introduce notation that allows us to state them simultaneously.

Let \(\pi\) have law \(\nu\), and
assume that \(\nu\) has density \(f_\pi\) on \((0,1)\) satisfying, for
some \(\beta>0\) and some slowly varying function
\(\ell:(0,\infty)\to(0,\infty)\),
\begin{equation}\label{eq:fpi-edge}
  f_\pi(u)\sim (1-u)^{\beta-1}\,\ell\!\left(\frac{1}{1-u}\right),
  \qquad u\uparrow 1.
\end{equation}

Integrating \eqref{asymtau} with respect to the law \(\nu\) gives
\[
  \PP(D^\rightarrow\ge n)
  =\int_{(0,1)}\E\bigl(u^{\tau_n^+}\bigr)\,\nu(du).
\]
Under the density assumption above, for each \(n\ge1\) define
\[
  A_n:=\int_0^1 n\,\E\bigl(u^{\tau_n^+}\bigr)\,f_\pi(u)\,du
  =n\,\PP(D^\rightarrow\ge n).
\]

Throughout this subsection, the increments retain the standing symmetry
assumption from Section~\ref{modeln}. We shall consider the following two
regimes:
\begin{itemize}
  \item[(a)] the heavy-jump regime, where
  \begin{equation}\label{eq:tail-xi}
    \PP(|\xi_1|>x)\sim c_\xi x^{-\alpha},
    \qquad x\to\infty,
  \end{equation}
  for some \(\alpha\in(1,2)\) and \(c_\xi>0\). Equivalently, under the
  symmetry assumption, the law of \(\xi_1\) belongs to the normal domain of
  attraction of a symmetric strictly \(\alpha\)-stable law;

  \item[(b)] the finite-variance regime, where
  \[
    \E[\xi_1]=0,
    \qquad
    \Var(\xi_1)=\sigma^2\in(0,\infty),
  \]
  for some \(\sigma>0\).
\end{itemize}
Accordingly, we write
\[
r=
\begin{cases}
\alpha, & \text{in regime \textup{(a)}},\\
2, & \text{in regime \textup{(b)}}.
\end{cases}
\]
In regime~\textup{(a)}, let $X_\xi$ denote the symmetric strictly
$\alpha$-stable Lévy process with Lévy measure
\[
\Lambda_\xi(dx)=\frac{\alpha}{2}c_\xi|x|^{-\alpha-1}\,dx.
\]
Its scale therefore depends on $c_\xi$. Let $B=(B_t)_{t\ge0}$ be the
standard Brownian motion. We write $\Gamma$ for the Gamma function.

\begin{proposition}
\label{prop:An-order}
Assume \eqref{eq:fpi-edge}. Then, in each of the two regimes
\textup{(a)} and \textup{(b)} above, if \(0<\beta<1\), one has
\begin{equation}\label{eq:An-order-unified}
 \lim_{n\to \infty} \frac{A_n}{n^{1-r\beta}\,\ell(n^r)}=
\begin{cases}
\Gamma(\beta)\,\E\!\big[(\tau_1^+(X_\xi))^{-\beta}\big],
& \text{in regime \textup{(a)}},\\[1mm]
2\,\Gamma(2\beta)\big(\frac{\sigma^2}{2}\big)^\beta,
& \text{in regime \textup{(b)}}.
\end{cases}
\end{equation}
\end{proposition}

\begin{theorem}
\label{thm:survival-extinction}
Assume $0<\E[\eta_0]<\infty$ and \eqref{eq:fpi-edge}. Define the motion-dependent constant
\[
\kappa:=
\begin{cases}
\Gamma\!\left(\frac1\alpha\right)\E\!\bigl[(\tau_1^+(X_\xi))^{-1/\alpha}\bigr],
& \text{in regime \textup{(a)}},\\[1mm]
\sqrt{2}\,\sigma,
& \text{in regime \textup{(b)}}.
\end{cases}
\]
In the stable regime, $\kappa$ depends on both $\alpha$ and the tail
constant $c_\xi$ through the scale of $X_\xi$. More explicitly, if $X_0$
denotes the normalized stable process with Lévy measure
$\frac{\alpha}{2}|x|^{-\alpha-1}\,dx$, then self-similarity gives
\[
\kappa=c_\xi^{1/\alpha}\Gamma\!\left(\frac1\alpha\right)
\E\!\left[(\tau_1^+(X_0))^{-1/\alpha}\right].
\]
The following assertions hold in each of the two regimes
\textup{(a)} and \textup{(b)} above:
\begin{enumerate}
  \item[\textup{(i)}] If \(0<\beta<1/r\), then \(A_n\to\infty\). In particular, Theorem~\ref{thm:frog-criterion}\textup{(ii)}
  implies that \(\mathrm{FM}(\mathbb{Z},\pi,\eta,S)\) survives with positive probability.

  \item[\textup{(ii)}] If \(\beta=1/r\), then
  \[
    \lim_{n\to\infty}\frac{A_n}{\ell(n^r)}=\kappa.
  \]
  In particular, this is a critical regime. If
  \[
    \liminf_{n\to\infty}\ell(n^r)>
    \frac{1}{\E[\eta_0]\kappa},
  \]
  then Theorem~\ref{thm:frog-criterion}\textup{(ii)} yields survival with
  positive probability for \(\mathrm{FM}(\mathbb{Z},\pi,\eta,S)\). If
  \[
    \limsup_{n\to\infty}\ell(n^r)<
    \frac{1}{4\,\E[\eta_0]\kappa},
  \]
  then Theorem~\ref{thm:frog-criterion}\textup{(i)}, together with \eqref{compara_d}, implies that
  \(\mathrm{FM}(\mathbb{Z},\pi,\eta,S)\) dies out almost surely.

  \item[\textup{(iii)}] If \(\beta>1/r\), then \(A_n\to0\). Again, Theorem~\ref{thm:frog-criterion}\textup{(i)}, together with \eqref{compara_d},
  implies that \(\mathrm{FM}(\mathbb{Z},\pi,\eta,S)\) dies out almost surely.
\end{enumerate}
\end{theorem}

\begin{remark}
For $\beta=1/r$, the preceding theorem identifies the exact asymptotic of
$n\PP(D^\rightarrow\ge n)$, but the two firework comparisons yield only
sufficient criteria. The survival and extinction constants differ by a factor of
four. Thus the off-critical threshold $\beta_c=1/r$ is sharp, whereas the full
phase diagram on the critical line is not determined by the present argument.
\end{remark}

\begin{remark}
For the simple symmetric nearest-neighbour walk, $\sigma^2=1$ and interval
activation coincides with the usual frog-model activation rule. Hence
$\beta_c=1/2$ and $\kappa=\sqrt{2}$. At $\beta=1/2$, the sufficient conditions
in Theorem~\ref{thm:survival-extinction} become
\[
\liminf_{n\to\infty}\ell(n^2)>\frac{1}{\sqrt{2}\,\E[\eta_0]}
\quad\text{for survival},
\qquad
\limsup_{n\to\infty}\ell(n^2)<\frac{1}{4\sqrt{2}\,\E[\eta_0]}
\quad\text{for extinction}.
\]
This specialization agrees with the off-critical nearest-neighbour threshold in
\cite{CarvalhoMachado2025,JGF}.
\end{remark}

A schematic summary of the survival and extinction regimes established in
Proposition~\ref{prop:survival-alpha-le-one} and
Theorem~\ref{thm:survival-extinction} is given in
Figure~\ref{fig:diagrama_SW_alpha}. For $0<\alpha<1$, the model survives
with positive probability for every $\beta>0$. At $\alpha=1$, the same
conclusion holds under the additional assumption $L_\xi(n)\to\infty$.

In the stable regime $1<\alpha<2$, the curve
$\beta_c(\alpha)=1/\alpha$ is the sharp off-critical threshold: the model
survives with positive probability when $\beta<1/\alpha$ and becomes
extinct almost surely when $\beta>1/\alpha$. On the critical curve
$\beta=1/\alpha$, with $1<\alpha<2$, Theorem~\ref{thm:survival-extinction}
gives survival if
\[
\liminf_{n\to\infty}\ell(n^\alpha)
>
\frac{1}{\E[\eta_0]\kappa},
\]
and almost-sure extinction if
\[
\limsup_{n\to\infty}\ell(n^\alpha)
<
\frac{1}{4\E[\eta_0]\kappa}.
\]
The behavior in the remaining cases on this critical curve is not
determined by the present comparison arguments.

The finite-variance regime is represented separately in the figure by
$\mathrm{FV}$, corresponding to the scaling index $r=2$. In this regime,
the model survives when $\beta<1/2$ and becomes extinct almost surely when
$\beta>1/2$. At $\beta=1/2$, since $\kappa=\sqrt{2}\sigma$, the sufficient
conditions become
\[
\liminf_{n\to\infty}\ell(n^2)
>
\frac{1}{\sqrt{2}\sigma\E[\eta_0]}
\]
for survival, and
\[
\limsup_{n\to\infty}\ell(n^2)
<
\frac{1}{4\sqrt{2}\sigma\E[\eta_0]}
\]
for almost-sure extinction; the remaining critical cases are again not
resolved by the present comparisons.

\begin{figure}[t]
  \centering
  \includegraphics[width=.98\textwidth]{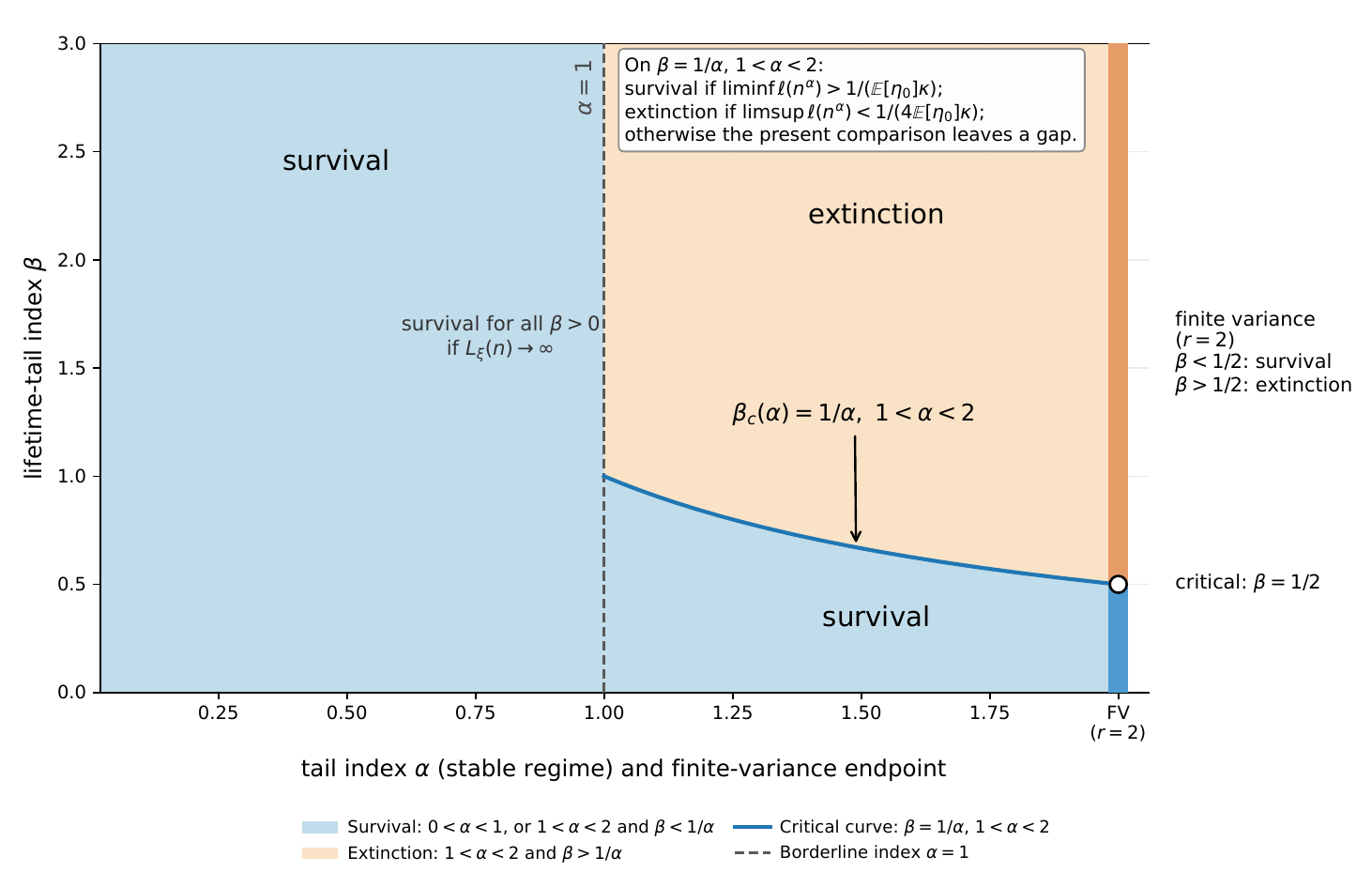}
    \caption{Schematic phase diagram for the survival, extinction, and
  unresolved critical regimes covered by
  Proposition~\ref{prop:survival-alpha-le-one} and
  Theorem~\ref{thm:survival-extinction}.}
  \label{fig:diagrama_SW_alpha}
\end{figure}

\begin{remark}

Suppose, more generally, that \eqref{eq:tail-xi-lattice0} holds for some \(\alpha>0\), with \(L_\xi\) slowly varying. Then
\[
\E[|\xi_1|^\rho]<\infty
\qquad\text{for every }0<\rho<\alpha,
\]
whereas
\[
\E[|\xi_1|^\rho]=\infty
\qquad\text{for every }\rho>\alpha.
\]
At the critical exponent \(\rho=\alpha\), the moment
\(\E[|\xi_1|^\alpha]\) may be either finite or infinite, depending on the slowly varying function \(L_\xi\).

Consequently, if $\alpha>2$, then $\E[\xi_1^2]<\infty$ and, under the
standing symmetry assumption, this case is included in the
finite-variance regime. In particular, the corresponding off-critical
threshold is $\beta_c=1/2$.

At the borderline value $\alpha=2$, the second moment may be either
finite or infinite, depending on the slowly varying factor. We give two
explicit symmetric integer-valued examples.

For the finite-variance example, define
\[
C_{\xi,1}
:=
\left(
2\sum_{k=2}^{\infty}\frac{1}{k^3(\log k)^2}
\right)^{-1}
\]
and let
\[
\PP(\xi_1=k)=\PP(\xi_1=-k)
=
\frac{C_{\xi,1}}{k^3(\log k)^2},
\qquad k\ge2.
\]
Then
\[
\PP(|\xi_1|\ge n)
=
2C_{\xi,1}
\sum_{k=n}^{\infty}\frac{1}{k^3(\log k)^2}
\sim
\frac{C_{\xi,1}}{n^2(\log n)^2},
\qquad n\to\infty,
\]
whereas
\[
\E[\xi_1^2]
=
2C_{\xi,1}
\sum_{k=2}^{\infty}\frac{1}{k(\log k)^2}
<\infty.
\]
Thus, this example has a regularly varying tail of index $-2$ and is
already covered by the finite-variance regime and, in particular, by
Lemma~\ref{thm:Donsker-Gaussian}.

For the infinite-variance example, define
\[
C_{\xi,2}
:=
\left(
2\sum_{k=1}^{\infty}\frac{1}{k^3}
\right)^{-1}
=
\frac{1}{2\zeta(3)}
\]
and let
\[
\PP(\xi_1=k)=\PP(\xi_1=-k)
=
\frac{C_{\xi,2}}{k^3},
\qquad k\ge1.
\]
In this case,
\[
\PP(|\xi_1|\ge n)
=
2C_{\xi,2}
\sum_{k=n}^{\infty}\frac{1}{k^3}
\sim
C_{\xi,2}n^{-2},
\qquad n\to\infty,
\]
but
\[
\E[\xi_1^2]
=
2C_{\xi,2}
\sum_{k=1}^{\infty}\frac{1}{k}
=
\infty.
\]
Hence, this example also has a regularly varying tail of index $-2$,
but it does not belong to the finite-variance regime.

The functional convergence of the rescaled random walk used in our
analysis is precisely the result stated in Lemma~\ref{Convgsl}. Its
proof is based on \cite[Theorem~1.1]{TyranKaminska2010SPA}, with
\cite[Corollary~1.2]{TyranKaminska2010SPA} used to verify the required
local-dependence condition. The results in that reference are formulated
for stable indices $\alpha\in(0,2)$ and therefore do not cover the
borderline value $\alpha=2$. Consequently, the infinite-variance
example above is covered neither by Lemma~\ref{thm:Donsker-Gaussian}
nor by Lemma~\ref{Convgsl}. Its treatment would require establishing a
separate functional limit theorem with the appropriate normalization
and then deriving the corresponding asymptotic behavior of the
first-passage time $\tau_n^+$.
\end{remark}

\section{Proofs}\label{proofmain}

In this section, we provide the proofs of the main results presented in Section~\ref{mainr}.

\begin{proof}[Proof of Theorem~\ref{thm:frog-criterion}]
We first record the link between dynamical survival and spatial percolation.
Because the common law $\nu$ is supported on $(0,1)$, every
$\pi_{x,i}<1$ almost surely. Hence every individual lifetime is finite almost
surely, and by countability this holds simultaneously for all frogs. Moreover, only
finitely many sites can be activated by any fixed time. Hence, if only finitely
many sites are ever activated, then only finitely many frogs are activated and
all of them die after a finite time. Conversely, once no active frogs remain, no further site can be activated. Therefore, survival is equivalent,
almost surely, to activation of infinitely many sites.

We now use the same environment to construct the firework comparisons. With
$N_x=\eta_x$ and $R_{x,i}=D_{x,i}^*$, the bidirectional process
$\mathrm{BFW}(I^*)$ contains every interval that can be activated by the frog
model, because
$[x-D_{x,i}^{\leftarrow},x+D_{x,i}^{\rightarrow}]\subseteq
[x-D_{x,i}^*,x+D_{x,i}^*]$. Iterating this inclusion from the origin over
successive activation generations shows that every site activated by the frog
model is informed by $\mathrm{BFW}(I^*)$. Thus
\[
\PP(\mathrm{BFW}(I^*)\text{ percolates})\ge
\PP\bigl(\mathrm{FM}(\mathbb Z,\pi,\eta,S)\text{ survives}\bigr).
\]
Before applying the firework criteria, we verify the standing nontriviality
condition in the firework setup of \cite{CarvalhoMachado2025}. For either
displacement radius $R\in\{D^*,D^\rightarrow\}$ of a single frog, let
$q_R:=\PP(R=0)$ and define the corresponding radius at site $x$ by
\[
I_x^{(R)}:=\max_{1\le i\le\eta_x}R_{x,i},
\]
with the maximum over the empty set equal to $0$. Since
$\{L=0\}\subseteq\{R=0\}$, one has $q_R\ge\E[1-\pi]>0$.
If $\PP(R>0)>0$, then $q_R<1$. Moreover, $0<\E[\eta_0]<\infty$ implies that $\eta_0<\infty$ almost surely and $\PP(\eta_0>0)>0$. Therefore the radius at the origin satisfies
\[
0<\PP(I_0^{(R)}=0)=\E[q_R^{\eta_0}]<1.
\]
For part~\textup{(i)}, the case $\PP(D^*>0)=0$ gives immediate extinction;
otherwise the preceding check, with $I_0^{(R)}=I_0^*$, permits the application
of \cite[Lemma~2.3(i)]{CarvalhoMachado2025}.

With $N_x=\eta_x$ and $R_{x,i}=D_{x,i}^{\rightarrow}$, the right-going
firework process informs the interval $[x,x+I_x^+]\cap\mathbb Z$. If
$\eta_x>0$, the maximum $I_x^+$ is attained by one of the finitely many frogs
at $x$, and that frog activates this entire interval under the interval rule;
if $\eta_x=0$, then $I_x^+=0$ and there is no nontrivial firework interval.
Iterating this inclusion over successive informed sites shows that every site
informed by $\mathrm{FW}(I^+)$ is activated by the frog model. Hence
\[
\PP(\mathrm{FW}(I^+)\text{ percolates})\le
\PP\bigl(\mathrm{FM}(\mathbb Z,\pi,\eta,S)\text{ survives}\bigr).
\]
The assumption in part~\textup{(ii)} implies $\PP(D^\rightarrow>0)>0$,
so the same calculation gives $\PP(I_0^+=0)\in(0,1)$. Hence
part~\textup{(ii)} follows from
\cite[Lemma~2.3(iii)]{CarvalhoMachado2025}.
\end{proof}

\begin{proof}[Proof of Proposition~\ref{prop:survival-alpha-le-one}]
The event $\{\tau_n^+=1\}=\{\xi_1\ge n\}$, together with survival through the
first attempted step, implies $D^\rightarrow\ge n$. Therefore,
\[
\PP(D^\rightarrow\ge n\mid\pi=p)\ge p\,\PP(\xi_1\ge n),
\]
and averaging over $\pi$ gives
\[
\PP(D^\rightarrow\ge n)\ge \E[\pi]\PP(\xi_1\ge n).
\]
Since \(0\le\pi<1\) and \(\PP(\pi>0)>0\), one has
\(0<\E[\pi]<1\). Moreover, by symmetry,
\[
\PP(\xi_1\ge n)
=
\frac12\PP(|\xi_1|\ge n).
\]
Consequently, under the hypothesis,
\[
n\PP(D^\rightarrow\ge n)
\ge
\frac{\E[\pi]}{2}\,
n\PP(|\xi_1|\ge n)
\longrightarrow\infty.
\]
The result follows from Theorem~\ref{thm:frog-criterion}\textup{(ii)}.
\end{proof}

\begin{proof}[Proof of Proposition~\ref{prop_regularly_alpha1}]
(i) Since
\(L\mid\pi=p
\)
is a geometric random variable with parameter $1-p$ supported on $\{0,1,\ldots\}$,
\begin{equation}\label{eq:esp_finita}
\E[L]
=
\E[\E[L\mid\pi]]
=
\E\!\left[\frac{\pi}{1-\pi}\right]
<
\infty.
\end{equation}

Observe that whenever a random walk is inside the interval
\(
[-n,n]\cap\mathbb{Z},
\)
a single jump whose absolute value is at least \(2n\) takes it to a
position whose absolute value is at least \(n\). Therefore, by
\eqref{compara_d},
\begin{equation}\label{eq:passo_2n}
\PP\left(\bigcup_{i=1}^{L}\{|\xi_i|\ge2n\}\right)
\le
\PP(D^*\ge n)
\le
2\PP(D^\rightarrow\ge n).
\end{equation}

Let
\(
p_n:=\PP(|\xi_1|\ge2n).
\)
Conditioning on $L$, we obtain
\(
\PP(\cup_{i=1}^{L}\{|\xi_i|\ge2n\}\mid L=k)
=
1-(1-p_n)^k,
\)
and therefore
\[
\PP\left(\bigcup_{i=1}^{L}\{|\xi_i|\ge2n\}\right)
=
\E\!\left[1-(1-p_n)^L\right].
\]

Furthermore,
\begin{equation}\label{conditions_dom_conv}
\frac{1-(1-p_n)^k}{p_n}\stackrel{p_n\downarrow 0}{\longrightarrow}k\, \text{ for all $k\in\mathbb{N}$} \qquad \text{and} \qquad \frac{1-(1-p_n)^L}{p_n}\le L,\end{equation}
The convergence follows directly from the binomial expansion of $(1-p_n)^k$, whereas the inequality is a consequence of the union bound. Combining \eqref{eq:esp_finita} and \eqref{conditions_dom_conv}, the Dominated Convergence Theorem yields
\[
\PP\left(\bigcup_{i=1}^{L}\{|\xi_i|\ge2n\}\right)
\sim
\E\!\left[\frac{\pi}{1-\pi}\right]
\PP(|\xi_1|\ge2n),
\qquad
n\to\infty.
\]

Hence, by \eqref{eq:passo_2n},
\[
\liminf_{n\to\infty}n\PP(D^\rightarrow\ge n)
\ge
\frac{\E[\frac{\pi}{1-\pi}]}{2}
\liminf_{n\to\infty}n\PP(|\xi_1|\ge 2n)
=
\frac{\E[\frac{\pi}{1-\pi}]}{4}
\liminf_{n\to\infty}L_\xi(n).
\]
Here the replacement of \(L_\xi(2n)\) by \(L_\xi(n)\) follows from
the fact that \(L_\xi\) is slowly varying.
The conclusion now follows from Theorem~\ref{thm:frog-criterion} \textup{(ii)}.

\medskip

(ii) Using once again the moments of the geometric distribution,
\begin{equation}\label{eq:esp_l2}
\E[L^2]
=
\E[\E[L^2\mid\pi]]
=
\E\!\left[
\frac{\pi+\pi^2}{(1-\pi)^2}
\right]
<
\infty.
\end{equation}

For $n\in\mathbb{N}$, define
\(
L'_\xi(n):=
n\PP(|\xi_1|\ge n).
\)
By assumption, there exist constants
\(
c_-,c_+\in(0,\infty)
\)
such that
\[
c_-
=
\liminf_{n\to\infty}L'_\xi(n)
=
\liminf_{n\to\infty}L_\xi(n)
\le
\limsup_{n\to\infty}L_\xi(n)
=
\limsup_{n\to\infty}L'_\xi(n)
=
c_+.
\]
Hence, there exists $n_0\in\mathbb{N}$ such that
\(
L'_\xi(n)\in[c_-/2,\,2c_+]\)
for every \(n>n_0.
\)
Moreover,
\(
0<
\PP(|\xi_1|\ge n_0)
\le
L'_\xi(n)
\le
n_0
<
\infty
\)
for every $n\le n_0$. Therefore, there exists an interval
\(
[a,b]\subset(0,\infty)
\)
such that
\(
L'_\xi(n)\in[a,b]
\)
 for every \(n\in\mathbb{N}.
\)
In particular, for every \(n,k\in\mathbb{N}\), since
\(|\xi_1|\) is integer-valued,
\begin{equation}\label{eq:cond_1}
\begin{aligned}
\PP(|\xi_1|\ge n/k)
&=
\PP\!\left(|\xi_1|\ge\left\lceil n/k\right\rceil\right) \\
&=
\frac{L'_\xi(\left\lceil n/k\right\rceil)}
{\left\lceil n/k\right\rceil}
\le
\frac{b}{\left\lceil n/k\right\rceil}
\le
\frac{bk}{n} \\
&\le
\frac{b}{a}\,k\,\frac{L'_\xi(n)}{n}
=
\frac{b}{a}\,k\,\PP(|\xi_1|\ge n).
\end{aligned}
\end{equation}
Since the tail of \(|\xi_1|\) is regularly varying with index
\(-1\), its distribution is subexponential. Therefore, by the
result stated immediately after Theorem~3.37 in \cite{FKZ},
together with conditions \eqref{eq:esp_l2} and
\eqref{eq:cond_1}, we obtain
\[
\PP\left(
\sum_{i=1}^{L}|\xi_i|
\ge
n
\right)
\sim
\E[L]\PP(|\xi_1|\ge n),\qquad n\to\infty.
\]

Consequently,
\[
\limsup_{n\to\infty}
n\PP(D^*\ge n)
\le
\limsup_{n\to\infty}
n
\PP\left(
\sum_{i=1}^{L}|\xi_i|
\ge
n
\right)
=
\E\!\left[\frac{\pi}{1-\pi}\right]
\limsup_{n\to\infty}
L_\xi(n).
\]
The result now follows from Theorem~\ref{thm:frog-criterion}\textup{(i)}.
\end{proof}

\begin{proof}[Proof of Proposition~\ref{prop:An-order}]
By Proposition~\ref{prop:tau-scaling},
\[
\frac{\tau_n^+}{n^r}\Rightarrow \tau_1^+(X_\xi)
\quad\text{in regime \textup{(a)}},
\qquad
\frac{\tau_n^+}{n^r}\Rightarrow \tau_1^+(\sigma B)
\quad\text{in regime \textup{(b)}}.
\]

By \eqref{eq:fpi-edge}, for every \(\varepsilon\in(0,1)\) there exists
\(\eta=\eta(\varepsilon)\in(0,1)\) such that
\begin{equation}\label{eq:fpi-two-sided}
  (1-\varepsilon)(1-u)^{\beta-1}\,
  \ell\!\left(\frac{1}{1-u}\right)
  \le f_\pi(u)\le
  (1+\varepsilon)(1-u)^{\beta-1}\,
  \ell\!\left(\frac{1}{1-u}\right),
  \quad u\in(1-\eta,1).
\end{equation}

Fix \(\varepsilon\in(0,1)\), and let \(\eta=\eta(\varepsilon)\in(0,1)\) be
given by \eqref{eq:fpi-two-sided}. By decreasing \(\eta\) if necessary, we
may assume that \(\eta<1/2\). We decompose
\[
A_n=A_n^{(1)}+A_n^{(2)},
\]
where
\[
A_n^{(1)}
:=\int_0^{1-\eta} n\,\E\bigl(u^{\tau_n^+}\bigr)\,f_\pi(u)\,du,
\qquad
A_n^{(2)}
:=\int_{1-\eta}^1 n\,\E\bigl(u^{\tau_n^+}\bigr)\,f_\pi(u)\,du.
\]

We first deal with \(A_n^{(1)}\). By Lemma~\ref{lem:moment-away-from-1},
there exist constants \(K_\eta<\infty\) and \(N_\eta\in\N\) such that
\[
\E\bigl(u^{\tau_n^+}\bigr)\le K_\eta n^{-r},
\qquad n\ge N_\eta,\quad u\in[0,1-\eta].
\]
Hence, for \(n\ge N_\eta\),
\[
0\le A_n^{(1)}
\le n\sup_{u\in[0,1-\eta]}\E\bigl(u^{\tau_n^+}\bigr)
   \int_0^{1-\eta}f_\pi(u)\,du
\le K_\eta n^{1-r}\int_0^{1-\eta}f_\pi(u)\,du
\le K_\eta n^{1-r}.
\]
Dividing by \(n^{1-r\beta}\ell(n^r)\), we obtain
\[
0\le
\frac{A_n^{(1)}}{n^{1-r\beta}\ell(n^r)}
\le
K_\eta\,
\frac{1}{n^{r-r\beta}\ell(n^r)}.
\]
Since \(0<\beta<1\), one has \(r-r\beta>0\), and therefore
\[
\frac{1}{n^{r-r\beta}\ell(n^r)}\xrightarrow[n\to\infty]{}0
\]
by slow variation of \(\ell\). Consequently,
\begin{equation}\label{eq:A1-limit-zero-final}
\lim_{n\to\infty}\frac{A_n^{(1)}}{n^{1-r\beta}\ell(n^r)}=0.
\end{equation}

We now turn to \(A_n^{(2)}\). Writing \(u=1-x\), we obtain
\[
A_n^{(2)}
=\int_0^\eta n\,\E\bigl((1-x)^{\tau_n^+}\bigr)\,f_\pi(1-x)\,dx.
\]
By \eqref{eq:fpi-two-sided},
\[
(1-\varepsilon)x^{\beta-1}\ell(1/x)
\le f_\pi(1-x)\le
(1+\varepsilon)x^{\beta-1}\ell(1/x),
\qquad x\in(0,\eta).
\]
Moreover, for \(x\in(0,\eta)\),
\[
x\le -\log(1-x)\le \frac{x}{1-\eta},
\]
and therefore
\[
e^{-x\tau_n^+/(1-\eta)}
\le (1-x)^{\tau_n^+}
\le e^{-x\tau_n^+}.
\]
Consequently,
\begin{equation}\label{eq:A2-two-sided-final-formal}
\begin{aligned}
&(1-\varepsilon)\int_0^\eta
n\,\E\bigl(e^{-x\tau_n^+/(1-\eta)}\bigr)\,
x^{\beta-1}\ell(1/x)\,dx \\
&\qquad\le A_n^{(2)} \\
&\qquad\le
(1+\varepsilon)\int_0^\eta
n\,\E\bigl(e^{-x\tau_n^+}\bigr)\,
x^{\beta-1}\ell(1/x)\,dx.
\end{aligned}
\end{equation}

For \(c>0\), define
\[
J_n(c):=
\int_0^\eta
n\,\E\bigl(e^{-cx\tau_n^+}\bigr)\,
x^{\beta-1}\ell(1/x)\,dx.
\]
With the change of variables \(x=y/n^r\), we obtain
\[
J_n(c)
=
n^{1-r\beta}
\int_0^{\eta n^r}
\E\!\left(e^{-cy\tau_n^+/n^r}\right)\,
y^{\beta-1}\ell\!\left(\frac{n^r}{y}\right)\,dy.
\]
Hence \eqref{eq:A2-two-sided-final-formal} becomes
\begin{equation}\label{eq:A2-sandwich-Jn-final-formal}
(1-\varepsilon)J_n\!\left(\frac{1}{1-\eta}\right)
\le A_n^{(2)}
\le (1+\varepsilon)J_n(1).
\end{equation}

We next determine the asymptotic behavior of \(J_n(c)\). Fix
\(\delta\in(0,\min\{\beta,1-\beta\})\) and \(A>1\). Since \(\ell\) is slowly
varying, Potter's bound \cite[Theorem~1.5.6]{BinghamGoldieTeugels1987} implies that, after decreasing \(\eta\) if necessary,
for all sufficiently large \(n\) and all \(y\in(0,\eta n^r]\),
\[
\frac{\ell(n^r/y)}{\ell(n^r)}
\le A\max\{y^\delta,y^{-\delta}\}.
\]
Thus
\[
\frac{J_n(c)}{n^{1-r\beta}\ell(n^r)}
=\int_0^\infty g_n(y)\,dy,
\]
where
\[
g_n(y):=
\mathbf{1}_{(0,\eta n^r]}(y)\,
\E\!\left(e^{-cy\tau_n^+/n^r}\right)\,
y^{\beta-1}\frac{\ell(n^r/y)}{\ell(n^r)}.
\]

For \(0<y\le1\), using \(\E(e^{-cy\tau_n^+/n^r})\le1\) and Potter's bound, we
obtain
\[
g_n(y)\le A\,y^{\beta-1-\delta},
\]
which is integrable near \(0\) because \(\delta<\beta\).

For large \(y\), fix \(t_0\in(0,1)\), let \(C_0>0\) be the constant from
Lemma~\ref{lem:fast-Tn}, and set
\[
A_0:=\max\{1,(ct_0)^{-1}\}.
\]
If \(y\ge A_0\), define
\[
F_n(t):=\PP\!\left(\frac{\tau_n^+}{n^r}\le t\right).
\]
Since $n\ge1$, one has $F_n(0)=0$. With the convention
$e^{-cy\infty}=0$, Stieltjes integration by parts therefore yields
\[
\E\!\left(e^{-cy\tau_n^+/n^r}\right)
=cy\int_0^\infty e^{-cyt}F_n(t)\,dt.
\]
Since $1/(cy)\le t_0$, Lemma~\ref{lem:fast-Tn} gives
\(F_n(t)\le C_0t\) for \(t\in(0,t_0]\), and therefore
\[
\begin{aligned}
\E\!\left(e^{-cy\tau_n^+/n^r}\right)
&\le cy\int_0^{t_0}e^{-cyt}F_n(t)\,dt
   +cy\int_{t_0}^\infty e^{-cyt}\,dt\\
&\le C_0cy\int_0^\infty te^{-cyt}\,dt+e^{-cyt_0}
 = \frac{C_0}{cy}+e^{-cyt_0}.
\end{aligned}
\]
Using Potter's bound once more, for \(y\ge A_0\),
\[
g_n(y)\le \frac{AC_0}{c}\,y^{\beta+\delta-2}
        +A\,e^{-cyt_0}y^{\beta+\delta-1}.
\]
This upper bound is integrable on $(A_0,\infty)$ because
$\beta+\delta<1$. On the compact interval $[1,A_0]$, Potter's bound and the
trivial estimate $\E(e^{-cy\tau_n^+/n^r})\le1$ give a finite constant bound
(independent of $n$). Together with the bound near zero, this provides an
integrable dominating function on $(0,\infty)$ for each fixed $c>0$. For
every fixed \(y>0\), Proposition~\ref{prop:tau-scaling} gives
\[
\E\!\left(e^{-cy\tau_n^+/n^r}\right)\longrightarrow
\begin{cases}
\E\!\left(e^{-cy\tau_1^+(X_\xi)}\right), & \text{in regime \textup{(a)}},\\[1mm]
\E\!\left(e^{-cy\tau_1^+(\sigma B)}\right), & \text{in regime \textup{(b)}}.
\end{cases}
\]
Since also
\[
\frac{\ell(n^r/y)}{\ell(n^r)}\longrightarrow 1,
\]
dominated convergence implies
\[
\lim_{n\to\infty}\frac{J_n(c)}{n^{1-r\beta}\ell(n^r)}
=
\begin{cases}
\displaystyle
\int_0^\infty \E\!\left(e^{-cy\tau_1^+(X_\xi)}\right)y^{\beta-1}\,dy,
& \text{in regime \textup{(a)}},\\[3mm]
\displaystyle
\int_0^\infty \E\!\left(e^{-cy\tau_1^+(\sigma B)}\right)y^{\beta-1}\,dy,
& \text{in regime \textup{(b)}}.
\end{cases}
\]
Making the change of variables \(t=cy\), we obtain
\[
\lim_{n\to\infty}\frac{J_n(c)}{n^{1-r\beta}\ell(n^r)}
=
c^{-\beta}
\begin{cases}
\displaystyle
\int_0^\infty \E\!\left(e^{-t\tau_1^+(X_\xi)}\right)t^{\beta-1}\,dt,
& \text{in regime \textup{(a)}},\\[3mm]
\displaystyle
\int_0^\infty \E\!\left(e^{-t\tau_1^+(\sigma B)}\right)t^{\beta-1}\,dt,
& \text{in regime \textup{(b)}}.
\end{cases}
\]

In regime \textup{(a)}, Lemma~\ref{lem:negative-moments-stable} gives
\[
\E\!\big[(\tau_1^+(X_\xi))^{-\beta}\big]<\infty
\qquad\text{for every }0<\beta<1,
\]
while in regime \textup{(b)}, Corollary~\ref{cor:BM-negative-moments} yields
\[
\E\!\big[(\tau_1^+(\sigma B))^{-\beta}\big]=\frac{2\,\Gamma(2\beta)}{\Gamma(\beta)}
    \Bigl(\frac{\sigma^2}{2}\Bigr)^\beta<\infty
\qquad\text{for every }\beta>0.
\]
Hence Tonelli's theorem and the Gamma identity imply
\begin{equation}\label{eq:Jnc-limit-final-formal}
\lim_{n\to\infty}\frac{J_n(c)}{n^{1-r\beta}\ell(n^r)}
=
c^{-\beta}C_\beta,
\end{equation}
where \[C_\beta=\begin{cases}
\Gamma(\beta)\,\E\!\big[(\tau_1^+(X_\xi))^{-\beta}\big],
& \text{in regime \textup{(a)}},\\[1mm]
2\,\Gamma(2\beta)
    \Bigl(\frac{\sigma^2}{2}\Bigr)^\beta,
& \text{in regime \textup{(b)}}.
\end{cases}\]

Applying \eqref{eq:Jnc-limit-final-formal} with \(c=1\) and
\(c=(1-\eta)^{-1}\), and using \eqref{eq:A2-sandwich-Jn-final-formal}, we
obtain
\[
(1-\varepsilon)(1-\eta)^\beta C_\beta
\le
\liminf_{n\to\infty}\frac{A_n^{(2)}}{n^{1-r\beta}\ell(n^r)}
\]
and
\[
\limsup_{n\to\infty}\frac{A_n^{(2)}}{n^{1-r\beta}\ell(n^r)}
\le
(1+\varepsilon)C_\beta.
\]
For every fixed admissible \(\eta\), equation
\eqref{eq:A1-limit-zero-final} shows that replacing \(A_n^{(2)}\) by the full
quantity \(A_n=A_n^{(1)}+A_n^{(2)}\) does not change either the limit inferior
or the limit superior. Hence
\[
(1-\varepsilon)(1-\eta)^\beta C_\beta
\le
\liminf_{n\to\infty}\frac{A_n}{n^{1-r\beta}\ell(n^r)}
\le
\limsup_{n\to\infty}\frac{A_n}{n^{1-r\beta}\ell(n^r)}
\le
(1+\varepsilon)C_\beta.
\]
For the fixed \(\varepsilon\), inequality \eqref{eq:fpi-two-sided} remains
valid for every smaller positive cutoff. We may therefore run the preceding
argument with arbitrarily small \(\eta\), let \(\eta\downarrow0\) in the lower
bound, and then let \(\varepsilon\downarrow0\). This gives
\[
\lim_{n\to\infty}\frac{A_n}{n^{1-r\beta}\ell(n^r)}=C_\beta,
\]
which is exactly \eqref{eq:An-order-unified}.
\end{proof}

\begin{proof}[Proof of Theorem~\ref{thm:survival-extinction}]
Assertions~\textup{(i)} and~\textup{(ii)}, and assertion~\textup{(iii)} when
$1/r<\beta<1$, follow directly from Proposition~\ref{prop:An-order}. It remains
to prove $A_n\to0$ for $\beta\ge1$.

Choose $\beta_1\in(1/r,1)$. By \eqref{eq:fpi-edge}, there exist
$\eta\in(0,1)$ and $C<\infty$ such that, for $u\in(1-\eta,1)$,
\[
f_\pi(u)\le C(1-u)^{\beta-1}\ell\!\left(\frac1{1-u}\right)
\le C(1-u)^{\beta_1-1}\ell\!\left(\frac1{1-u}\right).
\]
Decompose $A_n=A_n^{(1)}+A_n^{(2)}$ at $1-\eta$. Lemma~\ref{lem:moment-away-from-1}
gives $A_n^{(1)}=O(n^{1-r})=o(1)$.

Set
\[
w(u):=(1-u)^{\beta_1-1}\ell\!\left(\frac1{1-u}\right)
\mathbf 1_{(1-\eta,1)}(u),
\qquad a:=\int_0^1w(u)\,du.
\]
Potter's bound implies $0<a<\infty$. Thus $\widetilde f:=w/a$ is a density
satisfying
\[
\widetilde f(u)\sim a^{-1}(1-u)^{\beta_1-1}
\ell\!\left(\frac1{1-u}\right),\qquad u\uparrow1.
\]
Let $\widetilde A_n$ denote the quantity $A_n$ formed with the density
$\widetilde f$. Applying Proposition~\ref{prop:An-order} to $\widetilde f$
(and hence to the slowly varying factor $a^{-1}\ell$) gives
\[
\int_{1-\eta}^1n\E[u^{\tau_n^+}]w(u)\,du
=a\widetilde A_n
=O\!\left(n^{1-r\beta_1}\ell(n^r)\right)=o(1),
\]
because $r\beta_1>1$. Since
\[
A_n^{(2)}\le C\int_{1-\eta}^1n\E[u^{\tau_n^+}]w(u)\,du,
\]
it follows that \(A_n^{(2)}=o(1)\), and hence \(A_n\to0\).

For completeness, we now spell out the phase implications. In case~\textup{(i)},
$n\PP(D^\rightarrow\ge n)\to\infty$. Thus
Theorem~\ref{thm:frog-criterion}\textup{(ii)} gives survival with positive
probability. At criticality, the relation $A_n/\ell(n^r)\to\kappa$ gives the
stated survival criterion. For the extinction criterion, symmetry yields
\[
n\PP(D^*\ge n)\le 2A_n,
\]
so the stated upper bound on $\limsup_n\ell(n^r)$ implies
$\limsup_n n\PP(D^*\ge n)<1/(2\E[\eta_0])$ and
Theorem~\ref{thm:frog-criterion}\textup{(i)} applies. The same inequality,
together with $A_n\to0$, proves extinction throughout case~\textup{(iii)}.
\end{proof}

\section{Auxiliary results}\label{auxsec}

This section collects the auxiliary results used in the proof of the main theorems.
Throughout, we consider the random walk \(S=(S_n)_{n\ge0}\) on \(\mathbb{Z}\) defined by
\[
S_0:=0,
\qquad
S_n:=\sum_{k=1}^n \xi_k,\qquad n\ge1,
\]
where \((\xi_k)_{k\ge1}\) is an i.i.d.\ sequence of symmetric, integer-valued random variables.
We work under one of the following two assumptions on the increment law: either the
heavy-jump regime, in which
\[
\PP(|\xi_1|>x)\sim c_\xi x^{-\alpha},\qquad x\to\infty,
\]
for some \(\alpha\in(1,2)\) and \(c_\xi>0\), or the finite-variance regime, in which
\[
\E[\xi_1]=0,
\qquad
\Var(\xi_1)=\sigma^2\in(0,\infty).
\]
For \(n\in\mathbb{N}\), we write
\[
\tau_n^+:=\inf\{k\ge0:S_k\ge n\},
\]
with the convention \(\inf\varnothing=\infty\). More generally, for a càdlàg function
\(f:[0,\infty)\to\mathbb{R}\) and \(x>0\), we define
\[
\tau_x^+(f):=\inf\{t\ge0:f(t)\ge x\}.
\]
The purpose of this section is to establish the functional convergence of the corresponding rescaled walks and to derive the first-passage consequences needed later in the analysis.

\subsection{Functional convergence in Skorokhod space}\label{skoro}

For each \(T>0\), let \(D([0,T],\mathbb{R}^d)\) denote the space of all
\(\mathbb{R}^d\)-valued càdlàg functions on \([0,T]\); that is, functions
\(\psi:[0,T]\to\mathbb{R}^d\) that are right-continuous on \([0,T)\) and admit
finite left limits at every \(t\in(0,T]\).

Let \(\Lambda_T\) be the set of all strictly increasing continuous mappings
\(\lambda:[0,T]\to[0,T]\) satisfying \(\lambda(0)=0\) and \(\lambda(T)=T\).
For \(\psi_1,\psi_2\in D([0,T],\mathbb{R}^d)\), define
\[
d_T(\psi_1,\psi_2)
:=
\inf_{\lambda\in\Lambda_T}
\left\{
\sup_{0\le s\le T}\|\psi_1(\lambda(s))-\psi_2(s)\|
\vee
\sup_{0\le s\le T}|\lambda(s)-s|
\right\},
\]
where \(a\vee b:=\max\{a,b\}\). This is the Skorokhod \(J_1\) metric on
\(D([0,T],\mathbb{R}^d)\).

We also write \(D([0,\infty),\mathbb{R}^d)\) for the space of all càdlàg
functions \(\psi:[0,\infty)\to\mathbb{R}^d\). On this space we consider the
metric
\[
d_\infty(\psi_1,\psi_2)
:=
\int_0^\infty e^{-t}\bigl(d_t(\psi_1|_{[0,t]},\psi_2|_{[0,t]})\wedge 1\bigr)\,dt,
\qquad
\psi_1,\psi_2\in D([0,\infty),\mathbb{R}^d),
\]
where \(a\wedge b:=\min\{a,b\}\).

The metric \(d_\infty\) generates the Skorokhod \(J_1\) topology on
\(D([0,\infty),\mathbb{R}^d)\). In addition, both
\((D([0,T],\mathbb{R}^d),d_T)\) and
\((D([0,\infty),\mathbb{R}^d),d_\infty)\) are separable; see
\cite{Whitt1980} and \cite[Section~6]{JacodShiryaev2003}.

We shall also use the following criterion for weak convergence in
\(D([0,\infty),\mathbb{R}^d)\): if \(X_n\) and \(X\) are stochastic processes
with paths in \(D([0,\infty),\mathbb{R}^d)\), then
\[
X_n\Rightarrow X
\qquad\text{in }D([0,\infty),\mathbb{R}^d)
\]
if and only if
\[
X_n\Rightarrow X
\qquad\text{in }D([0,T],\mathbb{R}^d)
\]
for every \(T\in T_X:=\{t>0:\PP(X(t)\neq X(t-))=0\}\).

\begin{lemma}\label{Convgsl}
Let \((\xi_j)_{j\ge1}\) be an i.i.d.\ sequence of symmetric, integer-valued random
variables such that
\[
\PP(|\xi_1|>x)\sim c_\xi x^{-\alpha},
\qquad x\to\infty,
\]
for some \(\alpha\in(1,2)\) and \(c_\xi>0\). Set \(S_0:=0\),
\(S_k:=\sum_{j=1}^k \xi_j\) for \(k\ge1\), and define
\[
Y_{n,c_\xi}(t):=\frac{S_{\lfloor nt\rfloor}}{(n c_\xi)^{1/\alpha}},
\qquad t\ge0.
\]
Then
\[
Y_{n,c_\xi}\Rightarrow X
\qquad \text{in } D([0,\infty),\mathbb{R}),
\]
endowed with the Skorokhod \(J_1\) topology, where \(X=(X_t)_{t\ge0}\) is a
symmetric strictly \(\alpha\)-stable Lévy process on \(\mathbb{R}\) with
characteristic function
\[
\E\big[e^{i\theta X_t}\big]
=\exp\!\big(-t\,\sigma_{\alpha}|\theta|^\alpha\big),
\qquad \theta\in\mathbb{R},
\]
for
\[
\sigma_{\alpha}:=\Gamma(1-\alpha)\cos\!\Big(\frac{\pi\alpha}{2}\Big)>0.
\]
\end{lemma}

\begin{proof}
We establish the weak convergence of \(Y_{n,c_\xi}\) by applying
\cite[Theorem~1.1]{TyranKaminska2010SPA}. Since the sequence is i.i.d.,
its strong-mixing rate is identically zero. Symmetry also makes the truncated
centering term in that theorem vanish. It is therefore enough to verify the
local-dependence condition \(LD(\phi_0)\) and condition~(1.6) of that theorem,
which controls the accumulation of small jumps; the
remaining regular-variation assumptions follow directly from the two-sided
tail hypothesis: symmetry gives the spectral measure
$\frac12(\delta_{-1}+\delta_1)$, where $\delta_a$ denotes the unit point
mass at $a$, and the truncated centering is identically zero.

Set \(b_n:=(c_\xi n)^{1/\alpha}\), \(n\ge1\). In the notation of
\cite[(1.3)--(1.4)]{TyranKaminska2010SPA}, the centering constants are
\[
c_n=\frac{n}{b_n}\,
\E\!\left[\xi_1\mathbf 1_{\{|\xi_1|\le b_n\}}\right]=0
\]
by symmetry. Thus the partial-sum process in that theorem is exactly
\(Y_{n,c_\xi}\). The tail assumption in the statement gives
\[
n\,\PP(|\xi_1|>b_n)\to1
\qquad\text{as }n\to\infty.
\]
Since the tail of \(\xi_1\) is regularly varying with index
\(-\alpha\), for every \(\varepsilon>0\),
\[
\frac{\PP(|\xi_1|>\varepsilon b_n)}{\PP(|\xi_1|>b_n)}
\to \varepsilon^{-\alpha},
\qquad n\to\infty.
\]
Fix \(\varepsilon>0\) and \(k\ge1\). For all sufficiently large \(n\),
independence gives
\begin{equation*}
\begin{aligned}
&n\sum_{j=2}^{\lfloor n/k\rfloor}
\PP\bigl(
|\xi_j|>\varepsilon b_n,\,
|\xi_1|>\varepsilon b_n
\bigr)\\
&\qquad=
n\left(\left\lfloor\frac{n}{k}\right\rfloor-1\right)
\PP(|\xi_1|>\varepsilon b_n)^2\\
&\qquad=
\Bigl(n\,\PP(|\xi_1|>b_n)\Bigr)^2
\left(
\frac{\PP(|\xi_1|>\varepsilon b_n)}
     {\PP(|\xi_1|>b_n)}
\right)^2
\frac{\lfloor n/k\rfloor-1}{n}.
\end{aligned}
\end{equation*}
Consequently,
\[
n\sum_{j=2}^{\lfloor n/k\rfloor}
\PP\bigl(
|\xi_j|>\varepsilon b_n,\,
|\xi_1|>\varepsilon b_n
\bigr)
\longrightarrow
\frac{\varepsilon^{-2\alpha}}{k},
\qquad n\to\infty.
\]
It follows that
\[
\lim_{k\to\infty}\limsup_{n\to\infty}
n\sum_{j=2}^{\lfloor n/k\rfloor}
\PP\bigl(
|\xi_j|>\varepsilon b_n,\,
|\xi_1|>\varepsilon b_n
\bigr)
=0.
\]
Thus the hypotheses of \cite[Corollary~1.2]{TyranKaminska2010SPA} are satisfied, and the condition
\(LD(\phi_0)\) follows.
It remains to verify condition~(1.6). Fix
\(\delta>0\) and \(\varepsilon\in(0,1)\), and set
\[
  \zeta_{n,j}^{(\varepsilon)}
  :=\xi_j\mathbf 1_{\{|\xi_j|\le \varepsilon b_n\}}.
\]
Symmetry gives \(\E[\zeta_{n,j}^{(\varepsilon)}]=0\). Kolmogorov's maximal
inequality therefore yields
\[
\begin{aligned}
&\PP\!\left(
  \max_{1\le k\le n}
  \left|\sum_{j=1}^k\zeta_{n,j}^{(\varepsilon)}\right|
  \ge\delta b_n\right)\\
&\hspace{25mm}\le
  \frac{n}{\delta^2b_n^2}
  \E\!\left[\xi_1^2
    \mathbf 1_{\{|\xi_1|\le\varepsilon b_n\}}\right].
\end{aligned}
\]
We next justify the truncated-moment asymptotic and verify the hypotheses of
Karamata's theorem. Set
\[
Y:=|\xi_1|,
\qquad
\overline F(y):=\PP(Y>y).
\]
The tail assumption gives
$\overline F(y)\sim c_\xi y^{-\alpha}$, and hence
$y\overline F(y)$ is regularly varying with index $1-\alpha$. The function
$y\overline F(y)$ is measurable and locally bounded on $[1,\infty)$, while
\[
1-\alpha>-1
\]
because $\alpha<2$. Therefore Karamata's theorem
\cite[Theorem~1.5.11]{BinghamGoldieTeugels1987} yields
\[
\int_1^x y\overline F(y)\,dy
\sim \frac{x^2\overline F(x)}{2-\alpha}
\sim \frac{c_\xi}{2-\alpha}x^{2-\alpha}.
\]
Moreover, $y\overline F(y)\le y$ on $(0,1)$, so the integral over $(0,1)$ is
finite and does not affect the asymptotic. Finally, the tail-integration
identity
\[
\E\!\left[Y^2\mathbf 1_{\{Y\le x\}}\right]
=
2\int_0^x y\overline F(y)\,dy-x^2\overline F(x)
\]
follows by writing
$\min\{Y,x\}^2=Y^2\mathbf 1_{\{Y\le x\}}+x^2\mathbf 1_{\{Y>x\}}$.
Consequently,
\[
\E\!\left[\xi_1^2
  \mathbf 1_{\{|\xi_1|\le x\}}\right]
\sim
\left(\frac{2}{2-\alpha}-1\right)c_\xi x^{2-\alpha}
=
\frac{\alpha c_\xi}{2-\alpha}x^{2-\alpha},
\qquad x\to\infty.
\]
Since \(b_n^\alpha=c_\xi n\), it follows that
\[
\limsup_{n\to\infty}
\PP\!\left(
  \max_{1\le k\le n}
  \left|\sum_{j=1}^k\zeta_{n,j}^{(\varepsilon)}\right|
  \ge\delta b_n\right)
\le \frac{\alpha}{(2-\alpha)\delta^2}\,
   \varepsilon^{2-\alpha}.
\]
Letting \(\varepsilon\downarrow0\) proves condition~(1.6).

We have now verified all the assumptions of
\cite[Theorem~1.1]{TyranKaminska2010SPA}. Consequently, there exists an \(\alpha\)-stable
Lévy process \(X=(X_t)_{t\ge0}\) such that
\[
Y_{n,c_\xi}\Rightarrow X
\qquad\text{in }D([0,\infty),\mathbb{R}),
\]
endowed with the Skorokhod \(J_1\) topology.

It remains to identify the limit. For every \(x>0\), symmetry and regular
variation give
\[
 n\,\PP(\xi_1>b_nx)\longrightarrow \frac12x^{-\alpha},
 \qquad
 n\,\PP(\xi_1<-b_nx)\longrightarrow \frac12x^{-\alpha}.
\]
Thus the limiting Lévy measure has density
\[
\Lambda_\alpha(dx):=\frac{\alpha}{2}|x|^{-\alpha-1}\,dx,
\]
in agreement with the regular-variation description in
\cite[Section~2.3]{TyranKaminska2010SPA}.
Hence, by the Lévy--Khintchine formula,
\[
\E\bigl[e^{i\theta X_t}\bigr]
=
\exp\Bigg(
t\int_{\mathbb{R}}
\bigl(e^{i\theta x}-1-i\theta x\,\mathbf{1}_{\{|x|\le1\}}\bigr)\Lambda_\alpha(dx)
\Bigg),
\qquad \theta\in\mathbb{R}.
\]
Since \(\Lambda_\alpha\) is the Lévy measure of a symmetric strictly
\(\alpha\)-stable process, the exponent above is equal to
\(-t\,\sigma_\alpha|\theta|^\alpha\), where
\[
\sigma_\alpha:=\Gamma(1-\alpha)\cos\!\Bigl(\frac{\pi\alpha}{2}\Bigr)>0.
\]
Therefore,
\[
\E\bigl[e^{i\theta X_t}\bigr]
=
\exp\!\bigl(-t\,\sigma_\alpha|\theta|^\alpha\bigr),
\qquad \theta\in\mathbb{R},
\]
which completes the proof.
\end{proof}

\begin{lemma}\label{thm:Donsker-Gaussian}
Let $(\xi_j)_{j\geq 1}$ be an i.i.d.\ sequence of real-valued random variables such that
\[
\E[\xi_1]=0,
\qquad
\E[\xi_1^2]=\sigma^2\in(0,\infty).
\]
Set $S_0:=0$, $S_k:=\sum_{j=1}^k \xi_j$ for $k\ge1$, and define
\[
B_{n,\sigma}(t):=\frac{1}{\sigma\sqrt n}\,S_{\lfloor nt\rfloor},
\qquad t\ge0.
\]
Then, for every $T>0$,
\[
B_{n,\sigma}\big|_{[0,T]} \Rightarrow B
\qquad \text{in } D([0,T],\mathbb{R}),
\]
where $B$ is a standard Brownian motion on $[0,T]$. Moreover,
\[
B_{n,\sigma}\Rightarrow B
\qquad \text{in } D([0,\infty),\mathbb{R}).
\]
\end{lemma}

\begin{proof}
Fix $T>0$. By the classical Donsker invariance principle on compact
time intervals (see \cite[Theorem~14.1]{Billingsley1999}),
\[
B_{n,\sigma}\big|_{[0,T]} \Rightarrow B
\qquad \text{in } D([0,T],\mathbb{R}),
\]
where $B$ is a standard Brownian motion on $[0,T]$.

Let \(P_n\) and \(P\) denote the laws of \(B_{n,\sigma}\) and \(B\) on
\(D([0,\infty),\mathbb{R})\), respectively, and for each \(t>0\) let
\[
r_t:D([0,\infty),\mathbb{R})\to D([0,t],\mathbb{R}),
\qquad r_t(x):=x|_{[0,t]}.
\]
Then
\[
P_n\circ r_t^{-1}\Rightarrow P\circ r_t^{-1}
\qquad\text{for every }t>0.
\]
In addition, the Brownian motion admits a version with continuous
sample paths; see, e.g., \cite[Theorem~(1.9)]{RevuzYor1999}. Then
\[
P\bigl(B(t)\neq B(t-)\bigr)=0
\qquad\text{for every }t>0.
\]
Therefore, by \cite[Theorem~16.7]{Billingsley1999},
\[
P_n\Rightarrow P
\qquad\text{in }D([0,\infty),\mathbb{R}).
\]
Equivalently,
\[
B_{n,\sigma}\Rightarrow B
\qquad\text{in }D([0,\infty),\mathbb{R}).
\]
\end{proof}

\begin{remark}\label{ConvRe}
For \(c\in\R\), define \(\Phi_c:D([0,\infty),\R)\to D([0,\infty),\R)\) by
\[
\Phi_c(f):=cf.
\]
By the definition of \(d_\infty\), for every \(f,g\in D([0,\infty),\R)\),
\[
d_\infty(\Phi_c f,\Phi_c g)\le (|c|\vee 1)\,d_\infty(f,g).
\]
Hence \(\Phi_c\) is continuous on \(D([0,\infty),\R)\). Therefore, by the
continuous mapping theorem \cite[Theorem~2.7]{Billingsley1999},
if \(Y_n\Rightarrow X\) in \(D([0,\infty),\R)\), then
\[
cY_n=\Phi_c(Y_n)\Rightarrow \Phi_c(X)=cX
\qquad\text{in }D([0,\infty),\R).
\]

In particular, under the assumptions of Lemma~\ref{Convgsl}, if
\[
Y_n(t):=\frac{S_{\lfloor nt\rfloor}}{n^{1/\alpha}},\qquad t\ge0,
\]
then
\[
Y_n=c_\xi^{1/\alpha}Y_{n,c_\xi}\Rightarrow X_\xi:=c_\xi^{1/\alpha}X
\qquad\text{in }D([0,\infty),\R),
\]
where \(X\) is the limit in Lemma~\ref{Convgsl}. Consequently, \(X_\xi\) is a
symmetric strictly \(\alpha\)-stable Lévy process with Lévy measure
\begin{equation}\label{Levygene}
\Lambda_\xi(dx)=\frac{\alpha}{2}\,c_\xi\,|x|^{-\alpha-1}\,dx,
\end{equation}
and characteristic function
\[
\E\big[e^{i\theta X_{\xi,t}}\big]
=\exp\!\big(-t\,c_\xi\sigma_\alpha|\theta|^\alpha\big),
\qquad \theta\in\R,\ t\ge0.
\]

Likewise, under the assumptions of Lemma~\ref{thm:Donsker-Gaussian}, if
\[
B_n(t):=\frac{S_{\lfloor nt\rfloor}}{\sqrt n},\qquad t\ge0,
\]
then \(B_n=\sigma B_{n,\sigma}\) and therefore
\[
B_n\Rightarrow \sigma B
\qquad\text{in }D([0,\infty),\R),
\]
where \(B\) is a standard Brownian motion. Equivalently, the limit is a
Brownian motion with covariance function
\[
\E[(\sigma B_s)(\sigma B_t)]=\sigma^2\min\{s,t\}.
\]
\end{remark}

\subsection{First-passage times and functional convergence}

\begin{lemma}
\label{lem:limit-process-basics}
Let $X$ be either a Brownian motion with variance parameter $\sigma^2>0$ or a
symmetric strictly $\alpha$-stable Lévy process with $1<\alpha<2$ and Lévy
measure \eqref{Levygene}. For $x>0$ and $t>0$:
\begin{enumerate}
\item $\tau_x^>(X):=\inf\{s>0:X_s>x\}<\infty$ almost surely;
\item $S_X(t):=\sup_{0\le s\le t}X_s$ has an absolutely continuous law.
\end{enumerate}
\end{lemma}

\begin{proof}
In the Brownian case, recurrence implies that every positive level is crossed
almost surely; see \cite[Theorem~3.20]{MoertersPeres2010}. The reflection
principle gives an absolutely continuous law for $S_X(t)$; see
\cite[Chapter~III, Proposition~(3.7)]{RevuzYor1999}.

A symmetric strictly $\alpha$-stable process with $1<\alpha<2$ oscillates and is
unbounded above and below almost surely; see \cite[Chapter~VIII, Section~2]{Bertoin1996}.
Hence it crosses every positive level. By self-similarity,
$S_X(t)\stackrel d=t^{1/\alpha}S_X(1)$, and $S_X(1)$ has a continuous density by
\cite[p.~316]{DoneySavov2010}. This proves both assertions.
\end{proof}

\begin{lemma}\label{lem:tau-open-closed}
Under the assumptions of Lemma~\ref{lem:limit-process-basics}, for every $x>0$,
\[
\tau_x^+(X):=\inf\{t\ge0:X_t\ge x\}
=\tau_x^>(X)\qquad\text{almost surely}.
\]
\end{lemma}

\begin{proof}
For every $t>0$, the pathwise order
$\tau_x^+(X)\le\tau_x^>(X)$ and the right-continuity of the paths give
\[
\{S_X(t)<x\}
\subseteq
\{\tau_x^+(X)>t\}
\subseteq
\{\tau_x^>(X)>t\}
\subseteq
\{S_X(t)\le x\}.
\]
Indeed, if $S_X(t)<x$, then right-continuity at time $t$ implies that
$X_s<x$ on some interval $[t,t+\delta)$, and hence
$\tau_x^+(X)>t$. The last inclusion is immediate from the definition of
$\tau_x^>(X)$. Therefore,
\[
\PP(S_X(t)<x)
\le
\PP(\tau_x^+(X)>t)
\le
\PP(\tau_x^>(X)>t)
\le
\PP(S_X(t)\le x).
\]
By Lemma~\ref{lem:limit-process-basics},
$S_X(t)$ has an absolutely continuous law, and hence
\[
\PP(S_X(t)<x)=\PP(S_X(t)\le x).
\]
Consequently,
\[
\PP(\tau_x^>(X)>t)=\PP(\tau_x^+(X)>t),
\qquad t>0,
\]
so $\tau_x^+(X)$ and $\tau_x^>(X)$ have the same distribution.

We now make explicit why equality in distribution, together with the pathwise
order, implies almost-sure equality. If
$\PP(\tau_x^+(X)<\tau_x^>(X))>0$, then, by density and countability of
the positive rational numbers, there exists a rational number $q>0$ such that
\[
\PP\bigl(\tau_x^+(X)\le q<\tau_x^>(X)\bigr)>0.
\]
Since $\tau_x^+(X)\le\tau_x^>(X)$ almost surely, this would imply
\[
\PP(\tau_x^+(X)\le q)
>
\PP(\tau_x^>(X)\le q),
\]
contradicting equality in distribution. Therefore
\[
\PP\bigl(\tau_x^+(X)<\tau_x^>(X)\bigr)=0,
\]
and the pathwise order yields
\[
\tau_x^+(X)=\tau_x^>(X)
\qquad\text{almost surely}.
\]
\end{proof}
\begin{lemma}
\label{lem:tau-continuity}
Let $X$ satisfy Lemma~\ref{lem:limit-process-basics}, and let
$Y_n\Rightarrow X$ in $D([0,\infty),\mathbb R)$ with the Skorokhod $J_1$
topology. Assume that $Y_n(0)=0$ almost surely for every $n$. Then, for every
$x>0$,
\[
\tau_x^+(Y_n)\Rightarrow\tau_x^+(X).
\]
\end{lemma}

\begin{proof}
Fix $x>0$. We first establish convergence of the terminal suprema on an
arbitrary deterministic interval $[0,t]$, where $t>0$.

Both possible limit processes are L\'evy processes. For every fixed $t>0$, a
L\'evy process has no jump at time $t$ almost surely; see
\cite[Lemma~2.3.2]{Applebaum2009}. Thus
\[
\PP\bigl(\Delta X(t)\neq0\bigr)=0,
\qquad
\Delta X(t):=X_t-X_{t-}.
\]
In the Brownian case this also follows directly from continuity of the sample
paths. By the restriction criterion for weak convergence in
$D([0,\infty),\mathbb R)$, see
\cite[Theorem~16.7]{Billingsley1999}, the convergence
$Y_n\Rightarrow X$ therefore implies
\[
Y_n|_{[0,t]}\Rightarrow X|_{[0,t]}
\qquad\text{in }D([0,t],\mathbb R).
\]

Define the terminal-supremum functional
\[
\Phi_t:D([0,t],\mathbb R)\to\mathbb R,
\qquad
\Phi_t(f):=\sup_{0\le s\le t}f(s).
\]
This functional is continuous for the $J_1$ topology. Indeed, if
$f_n\to f$ in $D([0,t],\mathbb R)$ and $\lambda_n\in\Lambda_t$ are time
changes such that
\[
\sup_{0\le s\le t}|f_n(\lambda_n(s))-f(s)|\longrightarrow0,
\]
then $\lambda_n([0,t])=[0,t]$, and hence
\[
\begin{aligned}
\left|
\sup_{0\le s\le t}f_n(s)-\sup_{0\le s\le t}f(s)
\right|
&=
\left|
\sup_{0\le s\le t}f_n(\lambda_n(s))-
\sup_{0\le s\le t}f(s)
\right|\\
&\le
\sup_{0\le s\le t}|f_n(\lambda_n(s))-f(s)|
\longrightarrow0.
\end{aligned}
\]
The continuous mapping theorem, see
\cite[Theorem~2.7]{Billingsley1999}, now gives
\[
\sup_{0\le s\le t}Y_n(s)
\Rightarrow
S_X(t):=\sup_{0\le s\le t}X_s.
\]

For a general c\`adl\`ag path, the supremum on $[0,t]$ need not be attained.
For this reason one must use the deterministic inclusions
\[
\left\{\sup_{0\le s\le t}Y_n(s)<x\right\}
\subseteq
\{\tau_x^+(Y_n)>t\}
\subseteq
\left\{\sup_{0\le s\le t}Y_n(s)\le x\right\}.
\]
The first inclusion holds because, if
\[
\sup_{0\le s\le t}Y_n(s)<x,
\]
then \(Y_n(t)<x\). By right-continuity at time \(t\), there exists
\(\delta>0\) such that \(Y_n(s)<x\) for every
\(s\in[t,t+\delta)\), and hence \(\tau_x^+(Y_n)>t\).
The second inclusion holds because \(\tau_x^+(Y_n)>t\) implies that
\(Y_n(s)<x\) for every \(s\le t\), and therefore the supremum is at
most \(x\). 

By Lemma~\ref{lem:limit-process-basics}, $S_X(t)$ has an absolutely continuous
law, so $\PP(S_X(t)=x)=0$. Hence the preceding weak convergence and the
portmanteau theorem yield
\[
\PP\left(\sup_{0\le s\le t}Y_n(s)<x\right)
\longrightarrow
\PP(S_X(t)<x)
\]
and
\[
\PP\left(\sup_{0\le s\le t}Y_n(s)\le x\right)
\longrightarrow
\PP(S_X(t)\le x).
\]
The two limits coincide, and the squeeze theorem therefore gives
\[
\PP(\tau_x^+(Y_n)>t)
\longrightarrow
\PP(S_X(t)\le x),
\qquad t>0.
\]
Moreover, the probability squeeze established in the proof of
Lemma~\ref{lem:tau-open-closed} yields
\[
\PP(S_X(t)\le x)
=
\PP(\tau_x^>(X)>t),
\qquad t>0.
\]
Consequently,
\[
\PP(\tau_x^+(Y_n)>t)
\longrightarrow
\PP(\tau_x^>(X)>t),
\qquad t>0.
\]
We regard the first-passage times as random elements of
\(E:=[0,\infty]\), endowed with the one-point compactification topology.
Equivalently, let
\[
\phi:E\longrightarrow[0,1],\qquad
\phi(r)=\frac{r}{1+r}\ \text{for }r<\infty,\qquad
\phi(\infty)=1,
\]
and equip \(E\) with the metric
\(d_E(r,s):=|\phi(r)-\phi(s)|\). Then \(\phi\) is a homeomorphism from
\(E\) onto \([0,1]\).

By Lemma~\ref{lem:limit-process-basics} and right-continuity of the paths
of \(X\) at zero,
\[
\tau_x^>(X)\in(0,\infty)\qquad\text{almost surely}.
\]
Likewise, \(Y_n(0)=0<x\) and the paths of \(Y_n\) are right-continuous at
zero, so \(\tau_x^+(Y_n)>0\) almost surely for every \(n\).

For \(u\in(0,1)\), setting \(t=u/(1-u)\), the preceding tail convergence
gives
\[
\begin{aligned}
\PP\!\left(\phi\bigl(\tau_x^+(Y_n)\bigr)\le u\right)
&=1-\PP\!\left(\tau_x^+(Y_n)>\frac{u}{1-u}\right)\\
&\longrightarrow
1-\PP\!\left(\tau_x^>(X)>\frac{u}{1-u}\right)\\
&=\PP\!\left(\phi\bigl(\tau_x^>(X)\bigr)\le u\right).
\end{aligned}
\]
At \(u=0\), both distribution functions are equal to zero, whereas at
\(u=1\) both are equal to one; outside \([0,1]\), they are trivially
equal to zero or one. Hence
\[
\phi\bigl(\tau_x^+(Y_n)\bigr)
\Rightarrow
\phi\bigl(\tau_x^>(X)\bigr)
\qquad\text{in }[0,1].
\]
Since \(\phi^{-1}:[0,1]\to E\) is continuous, the continuous mapping
theorem \cite[Theorem~2.7]{Billingsley1999} yields
\[
\tau_x^+(Y_n)\Rightarrow\tau_x^>(X)
\qquad\text{in }[0,\infty].
\]
Finally, Lemma~\ref{lem:tau-open-closed} gives
\(\tau_x^>(X)=\tau_x^+(X)\) almost surely, and therefore
\[
\tau_x^+(Y_n)\Rightarrow\tau_x^+(X)
\qquad\text{in }[0,\infty].
\]

In particular, for every \(y>0\), the function
\[
r\longmapsto e^{-yr},\qquad r\in[0,\infty],
\]
with the convention \(e^{-y\infty}:=0\), is bounded and continuous on
\([0,\infty]\). Consequently, the preceding weak convergence implies
\[
\mathbb{E}\!\left[e^{-y\tau_x^+(Y_n)}\right]
\longrightarrow
\mathbb{E}\!\left[e^{-y\tau_x^+(X)}\right].
\]
\end{proof}

\begin{proposition}
\label{prop:tau-scaling} The following items hold:

\begin{enumerate}
  \item[(a)] Let \((\xi_j)_{j\ge1}\) be an i.i.d.\ sequence of symmetric, integer-valued random
variables such that
\[
\PP(|\xi_1|>x)\sim c_\xi x^{-\alpha},
\qquad x\to\infty,
\]
for some \(\alpha\in(1,2)\) and \(c_\xi>0\). Define \(Y_n(t):=\frac{S_{\lfloor n^\alpha t\rfloor}}{n},
\, t\ge0\). Then
  \[
  \frac{\tau_n^+}{n^\alpha}
  =\tau_1^+(Y_n)
  \Rightarrow \tau_1^+(X_\xi)
  \qquad\text{as }n\to\infty,
  \]
  where \(X_\xi\) is the symmetric strictly \(\alpha\)-stable Lévy process on
  \(\mathbb{R}\) with Lévy measure \(\Lambda_\xi\) given by \eqref{Levygene}.

  \item[(b)] Let $(\xi_j)_{j\geq 1}$ be an i.i.d.\ sequence of real-valued random variables such that
\[
\E[\xi_1]=0,
\qquad
\E[\xi_1^2]=\sigma^2\in(0,\infty).
\]
 Define \(Y_n(t):=\frac{S_{\lfloor n^2 t\rfloor}}{n},
\, t\ge0\). Then
  \[
  \frac{\tau_n^+}{n^2}
  =\tau_1^+(Y_n)
  \Rightarrow \tau_1^+(\sigma B)
  \qquad\text{as }n\to\infty,
  \]
  where \(B\) is a standard Brownian motion.
\end{enumerate}
\end{proposition}

\begin{proof}

For a function \(h\in D([0,\infty),\mathbb{R})\), let
\[
\Disc(h):=\{t>0:\ h(t)\neq h(t-)\}
\]
denote the set of discontinuity points of \(h\). Also, let
\[
C_Q([0,\infty),[0,\infty))
:=
\{g:[0,\infty)\to[0,\infty): g \text{ is continuous and strictly increasing}\}.
\]

In regime \textup{(a)}, let
\[
m_n:=\lfloor n^\alpha\rfloor,
\qquad
a_n:=\frac{m_n^{1/\alpha}}{n},
\qquad
\beta_n:=\frac{n^\alpha}{m_n},
\]
and define
\[
\widetilde Y_n(t):=\frac{S_{\lfloor m_n t\rfloor}}{m_n^{1/\alpha}},
\qquad
Y_n(t):=\frac{S_{\lfloor n^\alpha t\rfloor}}{n},
\qquad t\ge0.
\]
By Lemma~\ref{Convgsl} and Remark~\ref{ConvRe},
\[
\widetilde Y_n \Rightarrow X_\xi
\qquad\text{in }D([0,\infty),\mathbb{R}).
\]

For each \(n\ge1\), define
\[
G_n:D([0,\infty),\mathbb{R})\to D([0,\infty),\mathbb{R})
\]
by
\[
G_n(x)(t):=a_n\,x(\beta_n t),
\qquad t\ge0.
\]
Then, for every \(t\ge0\),
\[
G_n(\widetilde Y_n)(t)
=
a_n\,\widetilde Y_n(\beta_n t)
=
\frac{m_n^{1/\alpha}}{n}\,
\frac{S_{\lfloor m_n\beta_n t\rfloor}}{m_n^{1/\alpha}}
=
\frac{S_{\lfloor n^\alpha t\rfloor}}{n}
=
Y_n(t),
\]
because \(m_n\beta_n=n^\alpha\). We now prove that, whenever \(x_n\to x\) in \(D([0,\infty),\mathbb{R})\),
\[
G_n(x_n)\to x
\qquad\text{in }D([0,\infty),\mathbb{R}).
\]

Define
\[
g_n(t):=\beta_n t,
\qquad
e(t):=t,
\qquad
c_n(t):=a_n,
\qquad
\mathbf 1(t):=1,
\qquad t\ge0.
\]
Since \(\beta_n\to1\), for every \(T>0\),
\[
\sup_{0\le t\le T}|g_n(t)-e(t)|=|\beta_n-1|\,T\to0.
\]
Hence \(g_n\to e\) uniformly on compact intervals. Since each \(g_n\) and \(e\)
is continuous, and since the \(J_1\) topology on continuous functions coincides
with the topology of uniform convergence on compact intervals, it follows that
\[
g_n\to e
\qquad\text{in }D([0,\infty),[0,\infty));
\]
see \cite[p.~75]{Whitt1980}. Moreover,
\[
g_n,e\in C_Q([0,\infty),[0,\infty)).
\]
Therefore, by \cite[Theorem~3.1]{Whitt1980},
\[
x_n\circ g_n \to x\circ e=x
\qquad\text{in }D([0,\infty),\mathbb{R}),
\]
because the interval \([0,\infty)\) is open on the right.

Next, since \(a_n\to1\), for every \(T>0\),
\[
\sup_{0\le t\le T}|c_n(t)-\mathbf 1(t)|=|a_n-1|\to0.
\]
Hence \(c_n\to\mathbf 1\) uniformly on compact intervals, and thus
\[
c_n\to\mathbf 1
\qquad\text{in }D([0,\infty),\mathbb{R}).
\]
Since \(\mathbf 1\) is continuous,
\[
\Disc(\mathbf 1)=\varnothing.
\]
Therefore
\[
\Disc(x)\cap \Disc(\mathbf 1)=\varnothing.
\]
Hence, by \cite[Theorem~4.2]{Whitt1980},
\[
c_n(x_n\circ g_n)\to \mathbf 1\cdot x=x
\qquad\text{in }D([0,\infty),\mathbb{R}).
\]
But
\[
c_n(x_n\circ g_n)=G_n(x_n),
\]
so we have shown that
\[
G_n(x_n)\to x
\qquad\text{whenever }x_n\to x.
\]

We may now apply form~\textup{(ii)} of the continuous mapping theorem,
stated in \cite[p.~68]{Whitt1980}, to
\[
X_n:=\widetilde Y_n,\qquad X:=X_\xi,\qquad f_n:=G_n,\qquad f(x):=x.
\]
Since \(X_\xi\in D([0,\infty),\mathbb{R})\) almost surely, it follows that
\[
G_n(\widetilde Y_n)\Rightarrow X_\xi
\qquad\text{in }D([0,\infty),\mathbb{R}).
\]

Therefore, in regime \textup{(a)}, Lemma~\ref{lem:tau-continuity} with
\(x=1\) gives
\[
  \tau_1^+(Y_n)\Rightarrow \tau_1^+(X_\xi).
\]
In regime \textup{(b)}, Lemma~\ref{thm:Donsker-Gaussian} and
Remark~\ref{ConvRe} yield
\[
Y_n\Rightarrow \sigma B
\qquad\text{in }D([0,\infty),\mathbb{R}),
\]
and another application of Lemma~\ref{lem:tau-continuity} gives
\[
  \tau_1^+(Y_n)\Rightarrow \tau_1^+(\sigma B).
\]

To unify the notation in the two regimes, we write
\[
r=
\begin{cases}
\alpha, & \text{in the stable regime }1<\alpha<2,\\
2, & \text{in the finite-variance regime}.
\end{cases}
\]

Therefore, it remains only to verify that, for each \(n\ge1\),
\[
\tau_1^+(Y_n)=\frac{\tau_n^+}{n^r}.
\]

Indeed, if $\tau_n^+=\infty$,
then $S_k<n$ for all $k\ge0$, so $Y_n(t)<1$ for all $t\ge0$, and hence
$\tau_1^+(Y_n)=\infty$. If $\tau_n^+<\infty$ and $k:=\tau_n^+$, then
$S_k\ge n$ while $S_j<n$ for every $0\le j<k$. Therefore, for
$t<k/n^r$,
\[
\lfloor n^r t\rfloor\le k-1,
\]
and so $S_{\lfloor n^r t\rfloor}<n$, whereas at $t=k/n^r$ one has
\[
\lfloor n^r t\rfloor=k,
\]
and thus $S_{\lfloor n^r t\rfloor}\ge n$. This proves the identity.

Combining this identity with the two convergence statements above yields the
stated limits in \textup{(a)} and \textup{(b)}.
\end{proof}

\subsection{Properties of \texorpdfstring{\(\tau_x^+\)}{the first-passage time}}

In this subsection we study basic properties of the first-passage time
\(\tau_x^+\) for a Brownian motion on \(\mathbb{R}\) with variance parameter
\(\sigma^2>0\), and for a symmetric strictly \(\alpha\)-stable Lévy process on
\(\mathbb{R}\), with \(1<\alpha<2\), whose Lévy measure is given by
\eqref{Levygene}.

\begin{lemma}
\label{lem:small-time-limit}
Let \(X=(X_t)_{t\ge0}\) be either
\begin{itemize}
    \item a Brownian motion on \(\mathbb{R}\) with variance parameter \(\sigma^2>0\), or
    \item a symmetric strictly \(\alpha\)-stable Lévy process on \(\mathbb{R}\), with
    \(1<\alpha<2\), whose Lévy measure is given by \eqref{Levygene}.
\end{itemize}
Then there exist constants \(t_1\in(0,1)\) and \(C>0\) such that
\begin{equation}\label{eq:limit-small-linear}
  \PP(\tau_1^+(X)\le t)\le Ct
  \qquad\text{for all }t\in(0,t_1].
\end{equation}
Moreover, the distribution function
\[
  F(t):=\PP(\tau_1^+(X)\le t),\qquad t>0,
\]
is continuous on \((0,\infty)\).
\end{lemma}

\begin{proof}
We treat the two cases separately.

\emph{Stable case.}
Let \(S_X(t):=\sup_{0\le s\le t}X_s\), \(t\ge0\). Since \(X\) is
self-similar,
\[
  S_X(t)\stackrel d= t^{1/\alpha}S_X(1),
  \qquad t>0.
\]
By Lemma~\ref{lem:tau-open-closed},
\[
  \tau_1^+(X)=\tau_1^>(X)
  \qquad\text{almost surely}.
\]
Moreover, Lemma~\ref{lem:limit-process-basics} gives
\[
  \PP(S_X(t)=1)=0,
  \qquad t>0.
\]
Hence the probability squeeze established in the proof of
Lemma~\ref{lem:tau-open-closed} yields
\[
  \PP(\tau_1^>(X)>t)
  =
  \PP(S_X(t)\le1),
  \qquad t>0.
\]
Taking complements and using self-similarity, for every \(t>0\),
\[
  \PP(\tau_1^+(X)\le t)
  =
  \PP(\tau_1^>(X)\le t)
  =
  \PP(S_X(t)>1)
  =
  \PP(S_X(1)>t^{-1/\alpha}).
\]
By \cite[p.~316 and equation~(3)]{DoneySavov2010}, the random variable \(S_X(1)\)
has a continuous density and there exists a constant \(K>0\) such that
\begin{equation}\label{eq:Sup-tail-clean}
  \PP(S_X(1)>x)\sim Kx^{-\alpha},
  \qquad x\to\infty.
\end{equation}
Hence, substituting \(x=t^{-1/\alpha}\) into
\eqref{eq:Sup-tail-clean}, we obtain
\[
  \PP(\tau_1^+(X)\le t)
  \sim Kt,
  \qquad t\downarrow0.
\]
In particular, there exist \(t_1\in(0,1)\) and \(C>0\) such that
\[
  \PP(\tau_1^+(X)\le t)\le Ct
  \qquad\text{for all }t\in(0,t_1].
\]
Moreover, since \(S_X(1)\) has a continuous distribution and
\(t\mapsto t^{-1/\alpha}\) is continuous on \((0,\infty)\), the map
\(t\mapsto \PP(\tau_1^+(X)\le t)\) is continuous on \((0,\infty)\).

\medskip

\emph{Brownian case.}
Write \(X_t=\sigma B_t\), where \(B\) is a standard Brownian motion.
Since Brownian paths are continuous, Lemma~\ref{lem:tau-open-closed}
again yields
\[
  \tau_1^+(X)=\tau_1^>(X)
  \qquad\text{almost surely}.
\]
By the reflection principle (see, e.g., \cite[Chapter~III, Proposition~(3.7)]{RevuzYor1999}), for every \(t>0\),
\[
  \PP(\tau_1^+(X)\le t)
  =\PP\!\Big(\sup_{0\le s\le t}X_s\ge 1\Big)
  =2\,\PP(X_t\ge 1).
\]
Since $X_t$ is Gaussian with mean $0$ and variance $\sigma^2t$, standard Gaussian tail bounds
(see, e.g., \cite[Theorem~1.2.6]{Durrett2019}) imply that
\[
  \PP(X_t\ge 1)
  \le \frac{\sigma\sqrt t}{\sqrt{2\pi}}
     \exp\!\Big(-\frac{1}{2\sigma^2 t}\Big),
  \qquad t>0.
\]
Consequently, there exist constants \(C_1,c>0\) such that
\[
  \PP(\tau_1^+(X)\le t)
  \le C_1\sqrt t\,e^{-c/t},
  \qquad t>0.
\]
Since
\[
  t^{-1/2}e^{-c/t}\longrightarrow 0
  \qquad\text{as }t\downarrow0,
\]
there exist \(t_1\in(0,1)\) and \(C>0\) such that
\[
  \PP(\tau_1^+(X)\le t)\le Ct
  \qquad\text{for all }t\in(0,t_1].
\]
Finally, \(t\mapsto \PP(\tau_1^+(X)\le t)\) is continuous on
\((0,\infty)\), since \(t\mapsto \PP(X_t\ge1)\) is continuous.
\end{proof}

To treat the stable and finite-variance regimes simultaneously in the next two
lemmas, set
\[
r=
\begin{cases}
\alpha, & \text{in the stable case }1<\alpha<2,\\
2, & \text{in the finite-variance case},
\end{cases}
\qquad
Y_n(t):=\frac{S_{\lfloor n^r t\rfloor}}{n},\quad t\ge0.
\]
By Proposition~\ref{prop:tau-scaling}, in the corresponding regime,
\[
\tau_1^+(Y_n)=\frac{\tau_n^+}{n^r},
\qquad n\in\mathbb{N}.
\]

\begin{lemma}\label{lem:fast-Tn}
There exists $C<\infty$ such that
\begin{equation}\label{eq:Tn-small-linear}
\PP(\tau_1^+(Y_n)\le t)\le Ct,
\qquad n\ge1,\quad 0<t\le1.
\end{equation}
Equivalently,
\begin{equation}\label{eq:tau-small-linear}
\PP(\tau_n^+\le tn^r)\le Ct,
\qquad n\ge1,\quad 0<t\le1.
\end{equation}
\end{lemma}

\begin{proof}
Let $m:=\lfloor tn^r\rfloor$. If $m=0$, then
$\PP(\tau_n^+\le tn^r)=0$, so the claim is immediate. Assume henceforth that
$m\ge1$. By Lévy's maximal inequality for sums of independent symmetric
random variables \cite[Proposition~2.3]{LedouxTalagrand1991},
\[
\PP(\tau_n^+\le tn^r)
\le\PP\left(\max_{1\le k\le m}|S_k|\ge n\right)
\le 2\PP(|S_m|\ge n).
\]

In the finite-variance regime, Chebyshev's inequality yields
\[
\PP(|S_m|\ge n)\le \frac{m\sigma^2}{n^2}\le \sigma^2t.
\]
In the stable regime, the regular-variation assumption implies that there
exist constants \(x_0\ge 1/2\) and \(C_1<\infty\) such that
\[
\PP(|\xi_1|>x)\le C_1x^{-\alpha},
\qquad x\ge x_0.
\]
By increasing \(C_1\), if necessary, the same bound holds for every
\(x\ge 1/2\).

We also claim that there exists \(C_2<\infty\) such that
\[
\E\!\left[\xi_1^2\mathbf 1_{\{|\xi_1|\le x\}}\right]
\le C_2x^{2-\alpha},
\qquad x\ge 1/2.
\]
Indeed, setting \(Z:=|\xi_1|\), the tail-integration inequality gives
\[
\E\!\left[Z^2\mathbf 1_{\{Z\le x\}}\right]
\le
2\int_0^x y\,\PP(Z>y)\,dy.
\]
Using \(\PP(Z>y)\le1\) on \((0,1/2)\) and the preceding tail bound on
\([1/2,x]\), we obtain
\[
\begin{aligned}
\E\!\left[Z^2\mathbf 1_{\{Z\le x\}}\right]
&\le
2\int_0^{1/2}y\,dy
+
2C_1\int_{1/2}^x y^{1-\alpha}\,dy  \\
&\le C_2x^{2-\alpha},
\end{aligned}
\]
where we used \(1<\alpha<2\).

Now define
\[
A_{m,n}:=
\left\{\max_{1\le j\le m}|\xi_j|>\frac n2\right\},
\qquad
\zeta_j:=\xi_j\mathbf 1_{\{|\xi_j|\le n/2\}}.
\]
On \(A_{m,n}^c\), one has \(\zeta_j=\xi_j\) for every \(1\le j\le m\),
and hence
\[
S_m=\sum_{j=1}^m\zeta_j.
\]
Therefore,
\[
\{|S_m|\ge n\}
\subseteq
A_{m,n}
\cup
\left\{
\left|\sum_{j=1}^m\zeta_j\right|\ge n
\right\}.
\]
Since the law of \(\xi_j\) is symmetric and the truncation is symmetric,
\(\E[\zeta_j]=0\). Moreover, the variables \(\zeta_1,\ldots,\zeta_m\)
are independent. Thus, by the union bound and Chebyshev's inequality,
\[
\begin{aligned}
\PP(|S_m|\ge n)
&\le
\PP(A_{m,n})
+
\PP\!\left(
\left|\sum_{j=1}^m\zeta_j\right|\ge n
\right)\\
&\le
m\PP(|\xi_1|>n/2)
+
\frac{1}{n^2}
\Var\!\left(\sum_{j=1}^m\zeta_j\right)\\
&=
m\PP(|\xi_1|>n/2)
+
\frac{m\E[\zeta_1^2]}{n^2}.
\end{aligned}
\]
Applying the two estimates above with \(x=n/2\), we obtain
\[
\PP(|S_m|\ge n)
\le
Cmn^{-\alpha}
+
\frac{Cm n^{2-\alpha}}{n^2}
\le
Cmn^{-\alpha}.
\]
Since \(m=\lfloor tn^\alpha\rfloor\le tn^\alpha\), it follows that
\[
\PP(|S_m|\ge n)\le Ct.
\]

Combining this estimate with the finite-variance case and Lévy's maximal
inequality, and absorbing its factor \(2\) into the constant \(C\), proves
\eqref{eq:tau-small-linear}. Finally,
\[
\tau_1^+(Y_n)=\frac{\tau_n^+}{n^r},
\]
and therefore \eqref{eq:Tn-small-linear} follows immediately.
\end{proof}

We next derive a uniform bound for \(\E(u^{\tau_n^+})\) when \(u\) stays a
fixed distance away from \(1\), using the convention \(u^\infty:=0\) for
\(0\le u<1\).

\begin{lemma}\label{lem:moment-away-from-1}
Let $\eta\in(0,1)$. Then there exist constants $K_\eta<\infty$ and
$N_\eta\in\mathbb{N}$ such that
\begin{equation}\label{eq:moment-away-from-1}
\E\bigl(u^{\tau_n^+}\bigr)\le K_\eta\,n^{-r}
\end{equation}
for all $n\ge N_\eta$ and all $u\in[0,1-\eta]$.
\end{lemma}
\begin{proof}
Fix \(\eta\in(0,1)\) and let \(u\in[0,1-\eta]\). Then
\[
  \E\bigl(u^{\tau_n^+}\bigr)
  =\sum_{k=1}^\infty u^k\,\PP(\tau_n^+=k)
  \le \sum_{k=1}^\infty u^k\,\PP(\tau_n^+\le k).
\]
Since
\[
  \tau_1^+(Y_n)=\frac{\tau_n^+}{n^r},
\]
we have
\[
  \PP(\tau_n^+\le k)
  =\PP\!\left(\tau_1^+(Y_n)\le \frac{k}{n^r}\right).
\]
Therefore,
\[
  \E\bigl(u^{\tau_n^+}\bigr)
  \le \sum_{k=1}^\infty u^k\,
      \PP\!\left(\tau_1^+(Y_n)\le \frac{k}{n^r}\right).
\]

Fix any \(t_0\in(0,1]\), and let \(C>0\) be the constant in
Lemma~\ref{lem:fast-Tn}. Set
\[
  K_n:=\lfloor t_0 n^r\rfloor.
\]
Split the upper bound as
\[
  \E\bigl(u^{\tau_n^+}\bigr)\le S_{1,n}+S_{2,n},
\]
where
\[
  S_{1,n}:=\sum_{k=1}^{K_n}u^k\,
  \PP\!\left(\tau_1^+(Y_n)\le \frac{k}{n^r}\right),
\]
and
\[
  S_{2,n}:=\sum_{k>K_n}u^k\,
  \PP\!\left(\tau_1^+(Y_n)\le \frac{k}{n^r}\right).
\]

For \(1\le k\le K_n\), one has \(k/n^r\le t_0\), and hence
Lemma~\ref{lem:fast-Tn} yields
\[
  \PP\!\left(\tau_1^+(Y_n)\le \frac{k}{n^r}\right)
  \le C\,\frac{k}{n^r}.
\]
Therefore,
\[
  S_{1,n}
  \le \frac{C}{n^r}\sum_{k=1}^{K_n}k\,u^k
  \le \frac{C}{n^r}\sum_{k=1}^\infty k\,u^k
  = \frac{C}{n^r}\,\frac{u}{(1-u)^2}.
\]

For the tail term, we simply use
\[
  S_{2,n}
  \le \sum_{k>K_n}u^k
  = \frac{u^{K_n+1}}{1-u}.
\]
Since \(u\in[0,1-\eta]\), we have \(1-u\ge\eta\), and therefore
\[
  S_{2,n}\le \frac{(1-\eta)^{K_n+1}}{\eta}.
\]
Because \(K_n\sim t_0 n^r\), the right-hand side decays
exponentially in \(n^r\). In particular, there exists \(N_\eta\in\N\)
such that
\[
  S_{2,n}\le \frac{1}{n^r}
  \qquad\text{for all }n\ge N_\eta.
\]

Combining the two bounds, for \(n\ge N_\eta\) we obtain
\[
  \E\bigl(u^{\tau_n^+}\bigr)
  \le \frac{C}{n^r}\,\frac{u}{(1-u)^2}
      + \frac{1}{n^r}.
\]
Since \(u\in[0,1-\eta]\), the function
\[
  u\longmapsto C\,\frac{u}{(1-u)^2}+1
\]
is bounded on \([0,1-\eta]\). Hence there exists \(K_\eta<\infty\) such that
\[
  \E\bigl(u^{\tau_n^+}\bigr)\le K_\eta\,n^{-r}
\]
for all \(n\ge N_\eta\) and all \(u\in[0,1-\eta]\), which proves
\eqref{eq:moment-away-from-1}.
\end{proof}

We now turn to negative moments of the limiting first-passage time.

\begin{lemma}
\label{lem:negative-moments-stable}
Let $X$ be a symmetric strictly $\alpha$-stable Lévy process with
$1<\alpha<2$ and Lévy measure \eqref{Levygene}. Then, for every
$0<\theta<1$,
\begin{equation}\label{eq:T-neg-moment}
\E[(\tau_1^+(X))^{-\theta}]<\infty.
\end{equation}
\end{lemma}

\begin{proof}
Lemmas~\ref{lem:limit-process-basics} and~\ref{lem:tau-open-closed} imply that
\[
\tau_1^+(X)<\infty
\qquad\text{almost surely}.
\]
Moreover, since \(X_0=0\) and the paths of \(X\) are right-continuous at
zero,
\[
\tau_1^+(X)>0
\qquad\text{almost surely}.
\]
Indeed, for almost every sample path there exists \(\delta>0\) such that
\[
|X_s|<1
\qquad\text{for all }0\le s<\delta,
\]
and hence the level \(1\) cannot be reached before time \(\delta\).

By Lemma~\ref{lem:small-time-limit}, there exist constants
\(t_0,C>0\) such that
\[
\PP\bigl(\tau_1^+(X)\le t\bigr)\le Ct,
\qquad 0<t\le t_0.
\]
The layer-cake identity gives
\[
\E\!\left[(\tau_1^+(X))^{-\theta}\right]
=
\int_0^\infty
\PP\!\left((\tau_1^+(X))^{-\theta}>y\right)\,dy.
\]
Splitting the integral at \(t_0^{-\theta}\), we obtain
\[
\begin{aligned}
\E\!\left[(\tau_1^+(X))^{-\theta}\right]
&\le
t_0^{-\theta}
+
\int_{t_0^{-\theta}}^\infty
\PP\!\left(
\tau_1^+(X)<y^{-1/\theta}
\right)\,dy.
\end{aligned}
\]
For \(y>t_0^{-\theta}\), one has \(y^{-1/\theta}<t_0\), and therefore
\[
\PP\!\left(
\tau_1^+(X)<y^{-1/\theta}
\right)
\le
\PP\!\left(
\tau_1^+(X)\le y^{-1/\theta}
\right)
\le
C y^{-1/\theta}.
\]
Consequently,
\[
\begin{aligned}
\E\!\left[(\tau_1^+(X))^{-\theta}\right]
&\le
t_0^{-\theta}
+
C\int_{t_0^{-\theta}}^\infty y^{-1/\theta}\,dy\\
&=
t_0^{-\theta}
+
\frac{C\theta}{1-\theta}\,t_0^{1-\theta}
<\infty,
\end{aligned}
\]
where the integral is finite because \(0<\theta<1\). This proves
\eqref{eq:T-neg-moment}.
\end{proof}

\begin{lemma}
\label{lem:BM-Laplace}
Let \(X=(X_t)_{t\ge0}\) be a Brownian motion started from zero with variance
parameter \(\sigma^2>0\), that is,
\[
\E[X_t]=0,
\qquad
\E[X_t^2]=\sigma^2t,
\qquad t\ge0.
\]
For \(x>0\), set
\[
\tau_x^+(X):=\inf\{t\ge0:X_t\ge x\}.
\]
Then, for every \(q>0\),
\begin{equation}\label{eq:BM-Laplace}
\E\!\left[e^{-q\tau_x^+(X)}\right]
=
\exp\!\left(-\frac{x\sqrt{2q}}{\sigma}\right).
\end{equation}
For related Brownian first-hitting-time formulas, see also \cite[Part~II, Chapter~1, Section~2]{BorodinSalminen2002}.
\end{lemma}

\begin{proof}
Fix \(q>0\) and set
\[
\lambda:=\frac{\sqrt{2q}}{\sigma}.
\]
Then
\[
q=\frac{\lambda^2\sigma^2}{2},
\]
and the process
\[
M_t
:=
\exp\!\left(
\lambda X_t-\frac{\lambda^2\sigma^2}{2}t
\right)
=
\exp(\lambda X_t-qt),
\qquad t\ge0,
\]
is a martingale. We use the standard exponential-martingale property of
Brownian motion and the optional sampling theorem for bounded stopping times;
see \cite[Chapter~1, Section~1.3.C, and Chapter~2]{KaratzasShreve1991}.

By Lemmas~\ref{lem:limit-process-basics} and~\ref{lem:tau-open-closed},
\[
\tau_x^+(X)<\infty
\qquad\text{almost surely}.
\]
For every \(t>0\), the stopping time
\[
\tau_x^+(X)\wedge t
\]
is bounded. Hence the optional sampling theorem gives
\[
1
=
\E[M_0]
=
\E\!\left[M_{\tau_x^+(X)\wedge t}\right].
\]
Since the paths of \(X\) are continuous,
\[
X_{\tau_x^+(X)}=x
\qquad\text{almost surely}.
\]
Therefore,
\[
\begin{aligned}
1
&=
\E\!\left[
e^{\lambda x-q\tau_x^+(X)}
\mathbf 1_{\{\tau_x^+(X)\le t\}}
\right]  \\
&\quad+
\E\!\left[
e^{\lambda X_t-qt}
\mathbf 1_{\{\tau_x^+(X)>t\}}
\right].
\end{aligned}
\]
On the event \(\{\tau_x^+(X)>t\}\), one has \(X_t<x\). Consequently,
\[
0
\le
\E\!\left[
e^{\lambda X_t-qt}
\mathbf 1_{\{\tau_x^+(X)>t\}}
\right]
\le
e^{\lambda x-qt},
\]
and the right-hand side converges to zero as \(t\to\infty\).

On the other hand, since
\[
\mathbf 1_{\{\tau_x^+(X)\le t\}}
\uparrow 1
\qquad\text{almost surely},
\]
the monotone convergence theorem gives
\[
\lim_{t\to\infty}
\E\!\left[
e^{\lambda x-q\tau_x^+(X)}
\mathbf 1_{\{\tau_x^+(X)\le t\}}
\right]
=
e^{\lambda x}
\E\!\left[e^{-q\tau_x^+(X)}\right].
\]
Letting \(t\to\infty\) in the preceding optional-sampling identity yields
\[
1
=
e^{\lambda x}
\E\!\left[e^{-q\tau_x^+(X)}\right].
\]
Thus
\[
\E\!\left[e^{-q\tau_x^+(X)}\right]
=
e^{-\lambda x}
=
\exp\!\left(-\frac{x\sqrt{2q}}{\sigma}\right),
\]
which proves \eqref{eq:BM-Laplace}.
\end{proof}

\begin{corollary}
\label{cor:BM-negative-moments}
Let \(X=(X_t)_{t\ge0}\) be a Brownian motion started from zero with variance
parameter \(\sigma^2>0\), and define
\(\tau_1^+(X):=\inf\{t\ge0:X_t\ge1\}\). Then, for every \(\theta>0\),
\begin{equation}\label{eq:BM-negative-moments}
\E\!\left[(\tau_1^+(X))^{-\theta}\right]
=
\frac{2\,\Gamma(2\theta)}{\Gamma(\theta)}
\left(\frac{\sigma^2}{2}\right)^\theta.
\end{equation}
In particular,
\[
\E\!\left[(\tau_1^+(X))^{-\theta}\right]<\infty
\qquad\text{for every }\theta>0.
\]
\end{corollary}

\begin{proof}
Since \(X_0=0<1\) and the paths of \(X\) are continuous at zero,
\[
\tau_1^+(X)>0
\qquad\text{almost surely}.
\]
By Lemma~\ref{lem:BM-Laplace}, for every \(q>0\),
\[
\E\!\left[e^{-q\tau_1^+(X)}\right]
=
e^{-\lambda_0\sqrt q},
\qquad
\lambda_0:=\frac{\sqrt2}{\sigma}.
\]

For every \(\theta>0\) and every \(z>0\), the Gamma-integral identity gives
\[
z^{-\theta}
=
\frac{1}{\Gamma(\theta)}
\int_0^\infty q^{\theta-1}e^{-qz}\,dq.
\]
Applying this identity with \(z=\tau_1^+(X)\), and then using Tonelli's
theorem, which applies because the integrand is nonnegative, we obtain
\[
\begin{aligned}
\E\!\left[(\tau_1^+(X))^{-\theta}\right]
&=
\frac{1}{\Gamma(\theta)}
\E\!\left[
\int_0^\infty
q^{\theta-1}e^{-q\tau_1^+(X)}\,dq
\right] \\
&=
\frac{1}{\Gamma(\theta)}
\int_0^\infty
q^{\theta-1}
\E\!\left[e^{-q\tau_1^+(X)}\right]\,dq \\
&=
\frac{1}{\Gamma(\theta)}
\int_0^\infty
q^{\theta-1}e^{-\lambda_0\sqrt q}\,dq.
\end{aligned}
\]
Making the change of variables
\[
y=\sqrt q,
\qquad
q=y^2,
\qquad
dq=2y\,dy,
\]
gives
\[
\begin{aligned}
\int_0^\infty
q^{\theta-1}e^{-\lambda_0\sqrt q}\,dq
&=
2\int_0^\infty
y^{2\theta-1}e^{-\lambda_0y}\,dy \\
&=
2\lambda_0^{-2\theta}\Gamma(2\theta).
\end{aligned}
\]
Consequently,
\[
\E\!\left[(\tau_1^+(X))^{-\theta}\right]
=
\frac{2\,\Gamma(2\theta)}{\Gamma(\theta)}
\lambda_0^{-2\theta}.
\]
Finally, since
\[
\lambda_0=\frac{\sqrt2}{\sigma},
\]
we have
\[
\lambda_0^{-2\theta}
=
\left(\frac{\sigma^2}{2}\right)^\theta.
\]
Substitution proves \eqref{eq:BM-negative-moments}.
\end{proof}

\begin{acks}
G.O.C. was supported by  CNPq (grant no. 408590/2024-6). F.P.M. and J.H.R.G. are
supported by FAPESP (grants 2023/13453-5 and 2025/03804-0, respectively).
\end{acks}

\end{document}